\documentclass{amsart}
\usepackage{amsfonts,amsmath,amssymb,color,esint,tikz}
\usepackage{enumerate}
\usepackage{comment}
\usepackage{graphicx}
\numberwithin{equation}{section}
\usepackage{enumitem}
\usepackage{csquotes}
\usetikzlibrary{cd}
\usepackage[babel=true,kerning=true]{microtype}

\usepackage{amsthm,leqno}
\usepackage{hyperref}

\newtheorem{theo}{Theorem}[section]
\newtheorem{lem}{Lemma}[section]
\newtheorem{defi}{Definition}[section]
\newtheorem{theodefi}{Theorem-Definition}[section]
\newtheorem{cl}{Claim}[section]
\newtheorem{prop}{Proposition}[section]

\newtheorem{rem}{Remark}[section]

\newcommand{\eps}{\varepsilon}
\newcommand{\D}{\mathbb{D}}
\newcommand{\R}{\mathbb{R}}
\newcommand{\C}{\mathbb{C}}
\newcommand{\Z}{\mathbb{Z}}

\newcommand{\dusubseteq}{\mathrel{
\rotatebox[origin=c]{45}{$\subsetneq$}}
}

\newcommand{\ddsubseteq}{\mathrel{
\rotatebox[origin=c]{-45}{$\subsetneq$}}
}
\begin{document}

\title[]{Asymptotics for minimizers of the Ginzburg-Landau energy with optimal regularity of the boundary data and applications}
\author{Paul Laurain, Romain Petrides} 

\address[Paul Laurain]{
Institut de Mathématiques de Jussieu, Université de Paris, Bâtiment Sophie Germain, Case 7052, 75205
Paris Cedex 13, France \& DMA, Ecole normale supérieure, CNRS, PSL Research University, 75005 Paris.}
\email{paul.laurain@imj-prg.fr}
\address[Romain Petrides]{Institut de Mathématiques de Jussieu, Université de Paris, Bâtiment Sophie Germain, 75205
Paris Cedex 13, France}
\email{romain.petrides@imj-prg.fr}

\maketitle

\begin{abstract} 
We perform the classical asymptotic analysis for the Ginzburg-Landau energy which originates from the celebrated paper by Bethuel, Brezis, Hélein for nonsmooth boundary data. More precisely, we give optimal regularity assumptions on the boundary curve of planar domains and Dirichlet boundary data on them. When the Dirichlet boundary data is the tangent vector field of the boundary curve, our framework allows us to define a natural energy minimizing frame for simply connected domains enclosing Weil-Petersson curves.
\end{abstract}

The so-called simplified Ginzburg-Landau functional on a domain $\Omega$
$$ E_{\eps}\left(u\right) = \frac{1}{2} \int_{\Omega} \left\vert \nabla u \right\vert^2 + \frac{1}{4\eps^2}\int_{\Omega}\left(1-\vert u\vert^2\right)^2$$
and all its variations were intensively studied after the celebrated paper by Bethuel, Brezis, Hélein (BBH) \cite{BBH2}, since energies of this type are useful to study phase transition problems. From this seminal work, a lot of attention has been given to the behavior of the critical points: see e.g.  Pacard and
Rivière \cite{PR},  Sandier and Serfaty \cite{SS}, Farina and Mironescu \cite{FM}, Millot and Pisante \cite{MP}, Ignat, Nguyen, Slastikov and Zarnescu \cite{INSZ} and the references therein.

One other well known usefulness of this energy is to build kinds of \enquote{Dirichlet energy minimizing} maps with singularities adapted to the strong topological constraints given by some geometrical problems. This is our central motivation. For instance, we cannot find any smooth nor $H^1$ \textit{unit} vector field extensions with finite energy of maps $\partial{\Omega} \to \mathbb{S}^1$ of any topological degree. 
Therefore, we have to relax the constraint on the set of admissible vector fields in the variational problem. Vector fields do not necessary have unit norm anymore, but as a compensation, we add a term to the energy which drastically penalizes vector fields that do not have unit norm as $\eps \to 0$. Classical works by \cite{BBH2}, \cite{Struwe94}, prove the convergence as $\eps \to 0$ of minimizers of $E_{\eps}$ to harmonic maps with singularities that are minimizers of the so-called renormalized energy. In the same spirit, studying the Allen-Cahn equation is a valuable technique to build new minimal surfaces, see \cite{Gua} and \cite{CM}.

In the current paper, we aim at weakening the known regularity assumptions on the boundary data for the asymptotic analysis as $\eps\to 0$ for minimizers of the Ginzburg-Landau energy $E_{\eps}$
on simply connected domains $\Omega$ and $u : \Omega \to \mathbb{R}^2$ satisfying a Dirichlet boundary condition $u = g$ on $\Gamma := \partial\Omega$ for some functions $g : \Gamma \to \mathbb{S}^1$ with a prescribed degree. Not only the weakenings hold for the regularity of the boundary data $g$ but also for the regularity of the domain $\Omega$. Thanks to Jerrad and Sandier, see \cite{Jerrard}, \cite{Sandier} and \cite{Sandiererratum}, energy estimates on the Ginzburg-Landau energy allow $H^{\frac{1}{2}}$ boundary data but to our knowledge, the classical BBH asymptotic analysis which originates from the paper \cite{BBH2} has never been completely done with domains $\Omega$ with lower regularity than Lipschitz.


Before stating our result, we would like to give some motivations for these weakenings of assumptions. 
One by-product of the BBH analysis is to produce $u_* :\Omega \rightarrow S^1$ which is harmonic with $u_*=g$ on the boundary (of course $u_*$ is singular if $\mathrm{deg}(g)\not =0$). In the special case $g=\tau$ being the tangent unit vector field of the boundary, one can easily produce a family of such harmonic maps by considering a uniformization of the domain. Indeed, see proposition \ref{propcanonicaluniformization}, if $f_a :\D \rightarrow \Omega$ is a biholomorphic map such that $f(0)=a\in \Omega$, then $u_a=\frac{\partial_{\theta} f}{\left\vert \partial_\theta f \right\vert}$ is a desired harmonic map. As we show in Theorem \ref{propWeylPeterssonframeenergy}, using the BBH renormalized energy, $u_*$ is the one that maximizes $\vert f_a'(0) \vert $. More precisely, the expression of the renormalized energy associated to a uniformization $f$ has the form
$$\int_{\mathbb{D}} \left\vert \frac{f''}{f'} \right\vert^2 + 2\pi \ln \left\vert f'(0) \right\vert,$$
which is closely related to the Loewner energy of the boundary curve, see section \ref{motivations} for more details. Hence our goal is to make the BBH asymptotic analysis for domains whose boundary curves have the minimal regularity that makes this energy finite. Those curves are known as Weil-Petersson curves: they are chord-arc curves whose unit tangent vector field lies in $H^{\frac{1}{2}}$ (see theorem-definition \ref{thbishop} due to Bishop). More details, on the geometric motivations of our result are given in section \ref{motivations}.\\

Weil-Petersson curves are not necessarily $C^1$ curves because there are non-continuous $H^{\frac{1}{2}}$ functions. However, contrary to Lipschitz curves they cannot have corners see section 2.1 of \cite{RohdeWang}. Nevertheless, it is important to remark that they {\it a priori} not bound Lipschitz domain, since the slow log-spiral, $t\mapsto te^{i\log(\log(1/t))}$, is a Weil-Petersson curve.

\begin{figure}[!h]
\includegraphics[scale=.5]{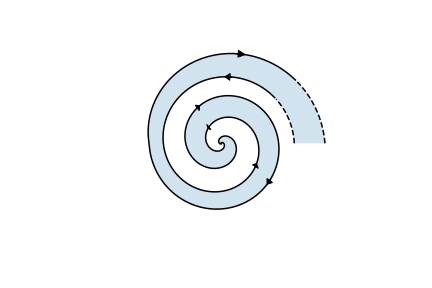}
\caption{A domain with slow log-spiral}
\end{figure}

Let's schematize all the regularity assumptions on the Jordan curve $\Gamma$ which borders $\Omega$ we mention in the paper (definitions are given in Section \ref{notdef}):

\[
\begin{array}{ccccc}
&&\text{Weil-Petersson curves} &&\\
&\dusubseteq& & \ddsubseteq&\\
\mathcal{C}^1 \text{ curves} &&&& \text{ Chord-arc curves} \subsetneq \text{ Quasicircles }\\
& \ddsubseteq&& \dusubseteq&\\
&&\text{Lipschitz curves}&&
\end{array}
\]

Coming back to the analysis, we first state an optimal result for the weakest regularity of $\Omega$ such that the classical definition of $H^{\frac{1}{2}}(\Gamma)$ makes sense:

\begin{theo} \label{theoGL}
Let $\Omega$ be a domain such that $\Gamma:= \partial\Omega$ is a chord-arc curve and $g \in H^{\frac{1}{2}}(\Gamma ,S^1)$ with $\mathrm{deg}(g)=1$, then there is $u_\eps \in H^{1}(\Omega)$ which minimizes, $E_{\eps}$
on the admissible set
$$ \mathcal{A} = \{ u\in H^1\left(\Omega,\mathbb{R}^2 \right) ; u = g \text{ a.e on }\Gamma   \} . $$
Moreover, for any sequence $\eps_k \to 0$, there are a subsequence $\eps_{k_l} \to 0$, $a\in \Omega$ and $u_{\star} \in C^\infty \left(\Omega \setminus \{a\}, S^1\right)\cap H^1_{loc} (\overline{\Omega}\setminus\{a\})$ such that

\begin{enumerate}[label=(\roman*)]
\item $u_{\star}$ is a harmonic map,
\item  $u_{\star}=g$ on $\partial \Omega$,
\item $\mathrm{deg}\left(\left(u_{\star}\right)_{\vert \partial B(a, \rho)}\right)=1$ for $\rho>0$ small enough,
\item $u_{\eps_{k_l}} \rightarrow u_{\star}$ in $C^\infty_{loc} (\Omega\setminus\{a\})$ as $l\to+\infty$,
\item $u_{\eps_{k_l}} \rightarrow u_{\star}$ in $H^1_{loc} (\overline{\Omega}\setminus\{a\})$ as $l\to +\infty$, 
\item $  \vert u_{\eps_{k_l}} \vert \to 1$ in $L^{\infty}_{loc}\left(\overline{\Omega}\setminus\{a\}\right)$ as $l\to +\infty$. 
\end{enumerate}
\end{theo}

This theorem can be easily generalized to the case $\mathrm{deg}(g) \in \Z^*$, since once the difficulty due to the weak regularity of the data at the boundary is overcome for one bad disk, then it is overcome for all. Hence, Theorems VI.1 and VI.2 of \cite{BBH} hold true in our setting. In the case $\mathrm{deg}(g) = 0$, it is even easier, since we do not need to remove any singularity and section \ref{secinteriorsingularity} is not necessary.  
The energy estimates on the Ginzburg-Landau energy given by Jerrad and Sandier allow $H^{\frac{1}{2}}$ boundary data and Lipschitz domains, as done in \cite{MRS}. 
For $\mathcal{C}^1$ domains with $H^{\frac{1}{2}}$ data at the boundary, the case $\mathrm{deg}(g)=0$ was simply solved in \cite{FM}. 
Notice that not only Theorem \ref{theoGL} extends the known results to much lower regularity but we also use and extend the natural and simple techniques due to the originated paper by BBH.

The assumptions of this theorem are optimal in the sense that the presence of the Dirichlet energy requires the $H^{\frac{1}{2}}$ condition on $g$, and $H^{\frac{1}{2}}\left(\Gamma\right)$ only makes sense if $\Gamma$ is a chord-arc curve. However, in the proof we only use that the pullback of $g$ by a uniformization of the domain belongs to $H^{\frac{1}{2}}(\mathbb{S}^1,\mathbb{S}^1)$. Hence, taking this as a definition of $H^{\frac{1}{2}}(\Gamma,\mathbb{S}^1)$ and defining  
$\mathcal{A} = \{ v\circ f^{-1}\} + H_0^1\left(\Omega,\mathbb{R}^2 \right)$, where $v$ is the harmonic extension of $g\circ f$ for some uniformization $f:\D \rightarrow \Omega$, we can even weaken the regularity of the boundary.
\begin{theo}
Let $\Omega$ be a quasidisk and $g \in H^{\frac{1}{2}}(\Gamma , \mathbb{S}^1)$ in the sense of definition \ref{H12Q}, then the conclusion of Theorem \ref{theoGL} holds.
\end{theo}

Indeed, one main idea to handle the low regularity of $\Omega$ is to pull back all the problem on the unit disk $\mathbb{D}$ thanks to one uniformization map $f : \mathbb{D}\to \Omega$. The unit disk is then a pleasant smooth domain, but we have to pay one price : the pullback functions $\tilde{u} = u \circ f$ for $u\in \mathcal{A}$ now belong to the $H^1$ space associated to the Riemannian metric $\left\vert f'(z) \right\vert^2 \left(dx^2+dy^2\right)$ which is the weighted space of functions having the following norm finite:

$$\left\|  u \right\|_{H^1\left(\Omega\right)}^2 = \left\| \tilde{u} \right\|_{H^1\left(\mathbb{D},  \left\vert f' \right\vert^2 \right)}^2 :=  \int_{\mathbb{D}} \tilde{u}^2 \left\vert f'(z) \right\vert^2  dz +  \int_{\mathbb{D}} \left\vert \nabla \tilde{u} \right\vert^2  dz  \hskip.1cm.$$

If $\Omega$ is a smooth domain, the weight $\left\vert f'(z) \right\vert^2$ satisfies a Harnack inequality \textit{up to the boundary}, but if $\Omega$ has low regularity, $\left\vert f'(z) \right\vert^2$ may go to $+\infty$ or $0$ as $\left\vert z \right\vert \to 1$ so that the geometry might be wildly different to the Euclidean one close to the boundary. 
In our setting, we assume that $\Omega$ is a quasidisk. 
It is a sufficient regularity for the domain in order to make the behavior of $\left\vert f'(z) \right\vert^2$ close to the boundary good enough to perform a BBH analysis on sequences of minimizers $u_{\eps}$. 
One reason is because the shape of disks $\mathbb{D}_r(x)$ centered in $x\in \Omega$ is not too deformed after the pullback by $f$, even when $x$ is close to $\Gamma$. This is because uniformization maps of quasidisks can be extended to quasiconformal maps on $\mathbb{C}$.

\medskip{}

For instance, one striking and fundamental lemma in the classical Ginzburg-Landau analysis is the uniform bound of $\eps \left\vert \nabla u_{\eps} \right\vert$ (see Step B.1 \cite{BBH2}) on $\Omega$. It is primordial to control the behavior of bad disks (small disks centered on points $x$ such that $\vert u_\eps(x) \vert \leq \frac{1}{2}$). Indeed, this estimate leads to the $\eta$-compactness (there is a quantum of energy around points such that $\vert u_\eps\vert \leq \frac{1}{2}$), see Lemma 24 of \cite{Ri}. Classically, this estimate is a consequence of the elliptic equation satisfied by $u_{\eps}$ as a critical point of $E_{\eps}$:
$$ - \Delta u_{\eps} = \frac{1}{\eps^2} \left(1-\left\vert u_{\eps} \right\vert^2 \right) u_{\eps} \text{ in }  \Omega  \hskip.1cm.$$
The argument relies on the bound $\left\|  u_{\eps} \right\|_{\infty} \leq 1$ coming from a maximum principle argument and then on a rescaling argument on the equation. However, with low regularity on $\Omega$, we cannot perform any maximum principle argument nor any rescaling argument close to the boundary. We first need to assume that $u_{\eps}$ is a minimizer of the energy in order to have the uniform bound $\left\|  u_{\eps} \right\|_{\infty} \leq 1$, and then, a rescaling argument only allows that $\eps \left\vert \nabla u_{\eps}(x) \right\vert $ is uniformly bounded on a set of points $x\in \Omega$ such that $\frac{d(x,\partial\Omega)}{\eps}$ is uniformly lower bounded. The main novelty of our paper (see Proposition \ref{propfarboundary}) is that there is $\eta_0>0$ such that for the points $x_{\eps} \in \Omega$ satisfying $d(x_{\eps},\partial\Omega) \leq \eta_0 \eps$, we have $\left\vert  u_{\eps}(x_{\eps}) \right\vert \geq \frac{1}{2}$. We crucially use the quasidisk assumption here. Thanks to this proposition, the classical \enquote{bad disks} cannot intersect the boundary of $\Omega$ and the classical steps of the BBH asymptotic analysis can be managed.

\medskip

Notice that Theorem \ref{theoGL} is only stated for simply connected domains and in Theorem \ref{propWeylPeterssonframeenergy} below, we give a link between the harmonic map $u_*$ and the uniformization of domains. However, our analysis provides all the tools for domains $\Omega$ such that $\partial \Omega$ is a disjoint union of chord-arc closed curves, and $H^{\frac{1}{2}}$ data of any degree on these closed curves. 
Then it would be interesting to give new constructions of uniformizations in the spirit of \cite{YDY}, \cite{JGHW}, see also \cite{Oudet}.

As already remarked, a uniformization $f : \mathbb{D}\to \Omega$ defined a harmonic map. Here we give a more precise description in terms of Coulomb frame. We set $p=f(0)$ and we let
$$ u = \frac{ f_{\theta}}{\left\vert  f_\theta \right\vert} \circ f^{-1} \text{ and } v = \frac{ f_r }{\left\vert  f_r \right\vert} \circ f^{-1} $$
the pushforward in $\Omega$ of the normalized angular and radial derivatives of $f$. Let $\omega =\langle d u ,v\rangle$ the Cartan form associated to $(v, u)$. We prove in Proposition \ref{propcanonicaluniformization} that this frame is Coulomb and that
$$ \star \omega = d\left(\mu +G_p \right) \hbox{ in } \Omega\hskip.1cm, $$
where $G_p$ is the Dirichlet Green function with respect to $p$ and that
$$ \left\vert \frac{f''}{f'} \right\vert^2 = \left\vert \nabla \left( \mu \circ f \right) \right\vert^2 \hbox{ in } \mathbb{D} \text{ and then } \int_{\mathbb{D}} \left\vert \frac{f''}{f'} \right\vert^2 dz = \int_{\Omega}  \left\vert \nabla \mu \right\vert^2 dz \hskip.1cm .$$

By definition, the last integral is finite if the boundary curve is Weil-Petersson. As briefly explained above, looking at the case $g=\tau$, where $\tau$ is the tangent vector field on $\Gamma$ such that the frame $(\nu, \tau)$ is a direct frame, where $\nu$ is the out-pointing normal of the bounded simply connected domain $\Omega$, the limit $u_*$ in Theorem \ref{theoGL} satisfies the following result, proved in section \ref{linkrenormalizeduniformization}:

\begin{theo} 
\label{propWeylPeterssonframeenergy} 
We assume that $\Gamma:=\partial \Omega$ is a Weil-Petersson curve. Let $u_* : \Omega\to \mathbb{S}^1$ be given by Theorem \ref{theoGL} with Dirichlet condition $u_*  = \tau$ on $\Gamma$. Then $u_* = u_a$ is the pushforward on $\Omega$ of normalized angular derivative of uniformization maps $f_a$ such that $f_a(0)=a$, where $a$ is the singularity of $u_*$. Moreover, the uniformization maps $f_a$ such that $f_a(0)=a$ maximize $\left\vert f'(0) \right\vert$ among all the uniformization maps $f$.
\end{theo}

From this theorem, we deduced that the renormalized energy of frames in a sense inspired by the BBH renormalized energy is closely related to the Loewner energy of the domain, see \cite{Wan} for a detailed presentation about the Loewner Energy.
However, in some sense, only a half part of the Loewner energy appears in our construction since the frame is only defined on $\Omega$. Therefore, in \cite{MW}, Michelat and Wang give some reinterpretation of the Loewner energy of a domain in terms of this new frame energy setting. Their work is complementary to ours.

\medskip

We also believe that our techniques are perfectly adapted to the analysis of the Ginzburg-Landau functional on surfaces. Some very interesting works are already given on closed immersed surfaces (see \cite{IJ}). In fact, the present article, with the one of \cite{MW}, are the first stones 
to define a renormalized frame energy in the case of surfaces with boundary but also with low regularity, ideally some weak immersion with boundary as the one defined by Rivière for closed surface, see \cite{Ri2}.

\medskip

\textbf{Acknowledgements :} This paper is part of a common project between Michelat and Wang and the two authors of the current paper, on the links between the Loewner energy and the Willmore energy. We would like to thank both of them for all the fruitfull discussions on the topic - more precisely, Yilin for her explanation of the Loewner Energy and Bishop's work, and Alexis for our conversations about Ginzburg-Landau - especially, Section \ref{secinteriorsingularity} owes him a lot.

\section{Motivations} \label{motivations}

By an idea coming from Chern \cite{Chern} (see also section 5.4 of \cite{HEL1996}), using moving frames in the Cartan formalism is more flexible in order to build conformal coordinates. This point of view is particularly interesting in higher dimension, where we can replace the absence of conformal coordinates by the existence of Coulomb frames. We briefly remind why searching an optimal frame gives rise to conformal coordinates in dimension $2$. The main idea consists in building a Coulomb frame in a neighborhood $U$ of a point on a surface $\Sigma$, i.e.  an orthonormal family $\vec{e}_1, \vec{e}_2 \in \Gamma(TU)$ such that 
\begin{equation}
\label{Coul}
d* (\langle \vec{e}_1 , d\vec{e}_2\rangle)=0.
\end{equation}
This could be done by minimizing the Dirichlet energy of the frame, assuming that locally the total curvature is small enough. Then we deduce from \eqref{Coul} that there is $\lambda \in C^\infty(U)$ such that 
\begin{equation}\label{Liebracket} [e^{\lambda} \vec{e}_1, e^{\lambda} \vec{e}_2]=0. \end{equation}
We obtain the desired conformal coordinates after integration.

It is important to note that {\it a priori}, this construction is purely local for two reasons. The first one is that if $\Sigma$ has some topology then the space of frames can be empty. The second is that to get some control on the frame we need that the total curvature does not exceed a certain level (see \cite{HEL1996} and \cite{Scha} for more details). Of course, on planar domains, there are no such problem. In this case, the minimizing Coulomb frame is given by the trivial frame but it doesn't see any geometry of the domain. However, since the boundary has a natural frame given by a tangent and a normal unit vector field, it seems interesting to consider it as Dirichlet boundary data for the Coulomb frame.

Let us consider a simply connected domain $\Omega$. A tangent unit vector field of $\Gamma := \partial\Omega$ can be seen as a map from $\Gamma $ to $\mathbb{S}^1$ of degree one. Then, it does not have any regular extension in $\Omega$ which still takes values into $\mathbb{S}^1$. Anyway, if we are able to construct a (singular) Coulomb frame with the tangent and exterior normal unit vector fields as Dirichlet boundary data, then by integrating it, we find a conformal map from the disk to our domain, which is nothing but a uniformization of the domain.

\begin{figure}[!h]
\includegraphics[page=1,scale=1.02]{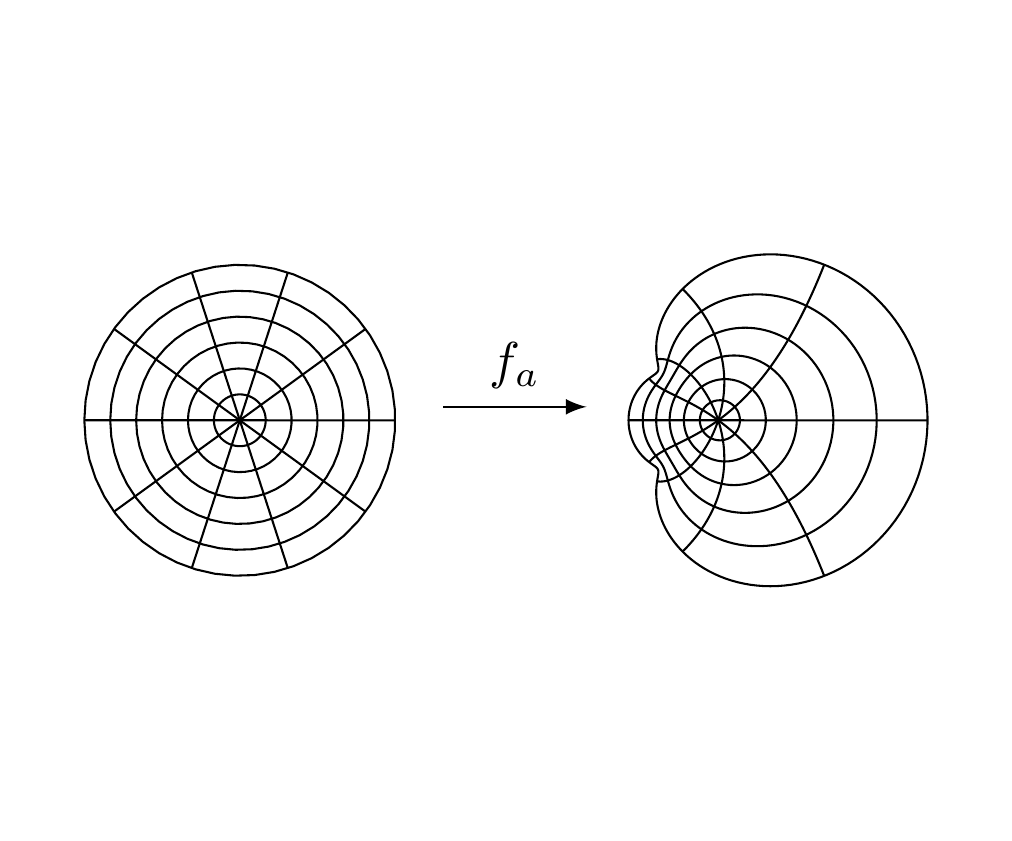}
\caption{A generic uniformization}
\end{figure}
\begin{figure}[!h]
\includegraphics[page=2,scale=1.02]{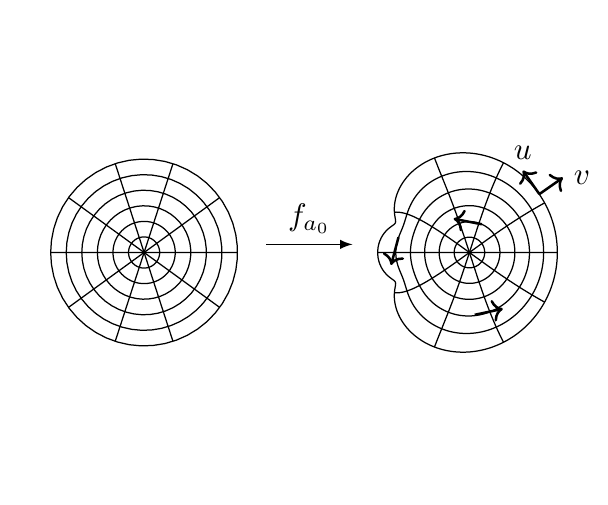}
\caption{The uniformazition given by the 'minimizing' frame}
\end{figure}

From this remark, we observe in proposition \ref{propcanonicaluniformization} that one natural \textit{global} moving frame $(v,u)$ is defined by the pushforward of the orthonormal polar coordinates on the disk by a uniformization map $f : \mathbb{D} \to \Omega$ defined by 
\begin{equation}
\label{uv}
 u = \frac{ f_{\theta}}{\vert f_{\theta}\vert} \circ f^{-1}  \text{ and } v = \frac{ f_{r}}{\vert f_{r}\vert} \circ f^{-1},
 \end{equation}
 where $f_\theta$ and $f_r$ denotes respectively the angular and radial derivatives of $f$. As expected, this frame has a singularity, located at $f(0)$. 
Note that for any choice of $a\in \Omega$, we can associate such a frame having the singularity $a$: it suffices to choose a uniformization $f : \mathbb{D}\to \Omega$ such that $f(0)=a$ by transitivity of the Mobius group in the disk. This uniformization is unique up to a rotation in the disk.
 The advantage of the frame approach is that since the frame is built by minimization then it should give a \enquote{best} Coulomb frame and hence a \enquote{best} uniformization. 

Nevertheless, as already pointed above, the space of smooth (even $H^1$) frames is empty by some degree obstruction. The problem of finding harmonic extensions of vector fields despite some topological obstruction is exactly the goal of the BBH analysis. Thanks to their seminal work, we know that on a smooth simply connected domain $\Omega$ there exists $u\in C^{\infty}(\Omega \setminus\{a_0\})$ such that $u$ is tangent to the boundary and $(-iu,u)$ is a coulomb frame. Moreover $a_0$ minimizes the renormalized energy 
$$ W(a) = \lim_{\delta\to 0}\left( \int_{\Omega \setminus \mathbb{D}_{\delta}(a)} \left\vert \nabla w \right\vert^2 - 2\pi \ln \frac{1}{\delta} \right),$$
where $w$ is the canonical harmonic vector field (see \cite{BBH}, Chapter I for details) matching with $u$ on the boundary and having its singularity at $a$. Hence among all the uniformizations, the \enquote{best} one is the one such that $f(0)=a_0$. In fact, $a_0$ is not necessary unique, but it is unique if the domain is not far from a disk, see \cite{CF} for the convex case.
Since we can conversely associate a harmonic vector field to some uniformization, thanks to \eqref{uv}, we can associate the following energy to each uniformization, see proposition \ref{renormalizedenergy}, 
\begin{equation}
\label{WPE}
W(f) = \int_{\mathbb{D}} \left\vert \frac{f''}{f'} \right\vert^2 + 2\pi \ln \left\vert f'(0) \right\vert .
\end{equation}
Then remarking that $$\int_{\mathbb{D}} \left\vert \frac{f''}{f'} \right\vert^2 +4\pi \ln \left\vert f'(0) \right\vert$$ is independent of the uniformization, the \enquote{best} one can be defined as the one that maximizes $\vert f'(0)\vert$. Of course, our method is not the first to determine the \enquote{best} conformal representation of a simply connected domain. For instance, in section 4.2 of \cite{Flu}, Flucher looks for an optimal conformal rearrangement by studying the Robin function of the domain. His motivation was the study of extremal functions for the Moser-Trudinger inequality. To our knowledge, it is the first time that the point of view of optimal frames is adopted, especially for domains with very low regularity. 

In our construction of a \enquote{best} uniformization presented above, we used the classical BBH analysis which is only valid for smooth domains. Having a careful look to the classical proofs, it seems that the regularity of the boundary can be reduced to $C^2$ but not much less since those proofs make use of the Pohozaev identity which requires a control on the derivative of the boundary data, that is to say here the derivative of the tangent unit vector field.

Hence, our main motivation is to be able to have the existence of the limit $u_{\star}$ for the minimizers of the Ginzburg-Landau functional as soon as \eqref{WPE} makes sense. In fact \eqref{WPE} is very similar to the Loewner energy of a Jordan curve $\Gamma$, given by
$$E_{L}(\Gamma)= \frac{1}{\pi}\left( \int_{\mathbb{D}} \left\vert \frac{f''}{f'} \right\vert^2 +\int_{\mathbb{D}} \left\vert \frac{g''}{g'} \right\vert^2 +4\pi \ln \left\vert \frac{f'(0)}{g'(\infty)}  \right\vert \right), $$
where $f:\D\rightarrow \Omega$, $g:\D^c \rightarrow \Omega^c$ are two uniformization with $\partial \Omega=\Gamma$. It is well-known that the class of domains to consider is the one enclosed by a Weil-Petersson curve (see next section for definition). Those curves are curves whose tangent unit vector field belongs to $H^\frac{1}{2}$.

Finally, to get the most general result for the asymptotic analysis of the Ginzburg-Landau energy, we decouple the regularity of the boundary $\partial \Omega$ and the regularity of the boundary data $g$ and we are able to prove the existence of a limit $u_*$ as soon as the boundary is chord-arc (not necessarily Weil-Petersson) and the boundary data $H^{\frac{1}{2}}$ in a classical sense. Moreover, giving a natural sense to the notion of $H^{\frac{1}{2}}$ maps for quasicircles (see next section), we are able to assume only that the boundary is a quasicircle. It is clear that this result is optimal, since the $H^{\frac{1}{2}}$ assumption cannot be weakened due to the presence of the Dirichlet energy, and Weil-Petersson curves are quasicircles. We then give the existence of a best uniformization in this weak setting.

To conclude, we want to notice that the uniformization was not exactly our first motivation. As explained in the introduction, this work has started by some discussions with Michelat and Wang on the way to rely quantitatively the Willmore energy to the Loewner Energy. To that purpose, we realized that we need the existence of a limit $u_*$ for the minimizers of the Ginzburg-Landau energy on a surface with a Weil-Petersson boundary. The present paper is then devoted to the case of a planar domains. It is interesting in itself since we have some purely BBH-type result with optimal regularity and a nice application to the problem of finding a best uniformization. From their side, Michelat and Wang have established some clear link of the frame approach and the Loewner-Energy. Indeed, they prove the existence of a consistent generalization of the frame energy \eqref{WPE} on $\Omega$ to a new frame energy defined on frames $(v_1,u_1)$ and $(v_2,u_2)$ defined on both sides of the curve $\Gamma$ viewed on the sphere by stereographic projection, denoted here by $W_{v,u}(\Gamma)$ (see \cite{MW} for a precise definition of this frame energy):
\begin{theo}[Michelat, Wang \cite{MW}, Theorem A]
Let $\Gamma \subset \mathbb{S}^2$ be a Weil-Petersson curve. Let $\Omega_1, \Omega_2$ be the two connected components of $\Omega \setminus \Gamma$. For $p_1 \in \Omega_1$ and $p_2 \in \Omega_2$, there are harmonic frames $(v_1,u_1) : \Omega_1\setminus \{p_1\} \to U\Omega_1$ and $(v_2,u_2) : \Omega_2\setminus \{p_2\} \to U\Omega_2$ such that for conformal maps $f_1 : \mathbb{D} \to \Omega_1 $ and $f_2 : \mathbb{D}\to \Omega_2$ such that $f_1(0)=p_1$ and $f_2(0) = p_2$, then
$$ E_{L}(\Gamma) = \frac{1}{\pi}W_{v,u}(\Gamma) + 2 \ln \left\vert \nabla f_1(0) \right\vert + 2 \ln \left\vert \nabla f_2 (0) \right\vert - 12 \ln(2),$$
where $W_{v,u}(\Gamma)$ is the renormalized energy of the frames $(v_1,u_1)$ and $(v_2,u_2)$.
\end{theo}

\section{Notations and definitions} \label{notdef}

\begin{itemize}[label={-}]
\item $\D$ denotes the open unit disk of the complex plane $\C$ identified with $\R^2$.
\item $\D_r=r \D$ and $\D_r(x)$ the open disk centered at $x$ with radius $r$. 
\item $S_r=\partial \D_r$
\item $A_{r_1,r_2}=\mathbb{D}_{r_2} \setminus \mathbb{D}_{r_1}$
\item $A^c$ is the complement of a given set $A$.
\end{itemize}

We remind some definition about quasiconformal maps, see \cite{Hu} for more details.
\begin{theodefi}[quasiconformal maps]
\label{KQCLQS}
Let $U,V$ two open sets of $\mathbb C$ and $K\geq 1$. A mapping $f:U\rightarrow V$ is $K$-quasiconformal if it is a homeomorphism and
\begin{enumerate}[label=\roman*)]
\item $f \in H^1_{loc}(U)$,
\item $\left\vert \frac{\partial f}{\partial \bar{z}} \right\vert \leq \frac{K-1}{K+1} \left\vert \frac{\partial f}{\partial z} \right\vert$ almost everywhere.
\end{enumerate}
It is equivalent to be $L$-quasisymmetric 
of modulus $\eta$ that is to say, there exists $\eta: \R^+ \rightarrow \R^+$ a homeomorphism (depending only on $K$) such that for any three distinct points $x,y,z \in U$ we have
$$ \left\vert \frac{f(x)-f(y)}{f(x)-f(z)}\right\vert \leq \eta \left( \left\vert\frac{x-y}{x-z}\right\vert\right).$$
\end{theodefi}

\begin{defi}[quasidisk, quasicircle]
A quasidisk is the image of $\D$ by a quasiconformal map from $\C$ into itself. A quasicircle is the boundary of a quasidisk.  
\end{defi}

Quasicircle can be very rough since it can even be a Julia set. Nevertheless, it can't be any Jordan curves since the cross-ratio must be bounded as shown by the following proposition

\begin{prop}
\label{qcp}
A Jordan curve $\Gamma \subset \mathbb C$ containing $1$ is a quasicircle if and only if there exists a constant $C$ such that for any $z_1,z_2,z_3 \in \Gamma$ that appears in that order\footnote{Which makes sense since $ 1 \in \Gamma$.} on $\Gamma$, we have 
\begin{equation}
\label{eqs}
\vert z_1 -z_2\vert \leq C \vert z_1 - z_3\vert .
\end{equation}
\end{prop}

We will make a central use of the following theorem which permits to extend uniformization of quasidisks to quasiconformal maps from $\C$ into itself.

\begin{theo}
\label{Extension}
Any uniformization map from $\D$ to a quasidisk admits a $K$-quasiconformal extension from $\C$ to itself.
\end{theo}

Now we are restricting our attention to a special class of quasicircles, the class of Weil-Petersson curves. Recently Bishop, see \cite{Bishop}, gives no less than 26 equivalent definitions for Weil-Petersson curves. The most adapted one to our setting is the following, which makes him said that 'roughly speaking, Weil-Petersson curves are to $L^2$ as quasicircles are to $L^\infty$'.
 
\begin{theodefi}[Definition 1 of \cite{Bishop}]
\label{thbishop} A Jordan curve $\Gamma$ enclosing a domain $\Omega$ is Weil-Petersson if it is a quasicircle and it admits a uniformization map $f :\D \rightarrow \Omega$ such that 
$$\int_\D \left\vert \frac{f''}{f'}\right\vert^2 \, dz <  \infty .$$
In particular it is rectifiable and a chord-arc curve, i.e. there exists $C$ such that for any subarc $\gamma\subset \Gamma$, we have
$$l(\gamma) \leq C \vert x-y\vert,$$
where $x,y$ are the end points of $\gamma$. 
\end{theodefi}

\begin{defi}[$H^\frac{1}{2}$ maps for chord-arc curves] Let $n\geq 1$ and $\Gamma$ be a chord-arc curves, in particular rectifiable. If $g :\Gamma \rightarrow \R^n$ satisfies
$$\int_\Gamma \int_\Gamma \left\vert \frac{g(x)-g(y)}{\vert x- y\vert} \right\vert^2 \, d\mu(x)\, d\mu (y) <\infty,$$
where $\mu$ is the measure of $\Gamma$, we say that $g \in H^\frac{1}{2}(\Gamma,\R^n)$ and we  denote $\Vert g \Vert_{H^{\frac{1}{2}}(\Gamma)}$ the left-hand side of the previous inequality. In particular, $g \in H^\frac{1}{2}(\Gamma,\mathbb{S}^1)$ if $g \in H^\frac{1}{2}(\Gamma,\R^2)$ and $g(x) \in \mathbb{S}^1$ for a.e. $x\in \Gamma$.
\end{defi}

In the previous definition we could have only assumed that $\Gamma$ is rectifiable, but we will need the following important proposition. It is due to Bishop, see page 16 of \cite{Bishop}, and it relies on the fact that quasisymmetric maps preserve the $H^{\frac{1}{2}}$-norm by Ahlfors and Beurling, see page 129 of \cite{AB}.

\begin{theo}
\label{H12} Let $\Gamma$ be a chord-arc curve enclosing a domain $\Omega$ and $f :\D \rightarrow \Omega$ be a uniformization, then there exists $C>0$ depending only on $\Gamma$, such that for every $g \in H^\frac{1}{2}(\Gamma,\R^n)$ we have
$$\frac{1}{C} \Vert g \Vert_{H^{\frac{1}{2}}(\Gamma)} \leq \Vert g\circ f \Vert_{H^{\frac{1}{2}}(\mathbb{S}^1)} \leq C \Vert g \Vert_{H^{\frac{1}{2}}(\Gamma)} .$$
\end{theo}

This theorem leads us to the following definition of $H^{\frac{1}{2}}\left(\Gamma,\mathbb{S}^1\right)$ maps 

\begin{defi}[$H^\frac{1}{2}(\Gamma, \mathbb{S}^1)$ maps]
\label{H12Q} Let $\Gamma$ be a quasicircle enclosing a domain $\Omega$, and $f:\D \rightarrow \Omega$ a uniformization. We say that $g :\Gamma \rightarrow \mathbb{S}^1$ is in $ H^\frac{1}{2}(\Gamma,\mathbb{S}^1)$ if $g\circ f \in H^\frac{1}{2}(\mathbb{S}^1,\mathbb{S}^1)$ and we set $\Vert g \Vert_{H^{\frac{1}{2}}(\Gamma)}=\Vert g\circ f \Vert_{H^{\frac{1}{2}}(\mathbb{S}^1)}$.
\end{defi}

Notice that there are subtleties with this definition. Indeed, let $v$ be the harmonic extension of $g\circ f \in H^\frac{1}{2}(\mathbb{S}^1,\mathbb{S}^1)$. Then, knowing both $v \in H^1(\mathbb{D},\mathbb{R}^2)$ and $v\circ f^{-1} \in H^1(\Omega,\mathbb{R}^2) $ makes our definition consistent. We explain why in the appendix \ref{appendixexplainH12}. 

\medskip

Now we define the degree of a weak maps, see \cite{BN} and the appendix of \cite{BGP}.

\begin{theo} Let $g\in H^{\frac{1}{2}}(\mathbb{S}^1,\mathbb{S}^1)$, then there exists an integer $n$ and $\phi \in H^{\frac{1}{2}}(\mathbb{S}^1,\R)$, unique up to a multiple of $2\pi$, such that $g=z^n e^{i\phi}$. The integer $n$ is called the degree of $g$, and is denoted $\mathrm{deg}(g)$, moreover we have\footnote{This formula is well defined since $g\in H^\frac{1}{2}$ and $g_\theta \in H^{-\frac{1}{2}}$.}
\begin{equation}
\label{eqdeg}
\mathrm{deg}(g)=\frac{1}{2\pi} \int_{\mathbb{S}^1}  g\wedge  g_\theta  \, d\theta.
\end{equation}
\end{theo} 

Thanks to Theorem \ref{H12}, we have

\begin{theodefi} For $\Gamma$ a chord-arc curve and $g\in H^\frac{1}{2}(\Gamma, \mathbb{S}^1)$, we define $\mathrm{deg}(g,\Gamma)=\mathrm{deg}(g\circ f)$, where $f$ is any uniformization map of the enclosing domain of $\Gamma$. 
\end{theodefi}

In particular, see property 5 and corollary 8 of \cite{BN}, we have

\begin{theo}
\label{somdeg}
Let $\Omega$ a domain whose boundary consists of a finite number of quasicircle $(\Gamma_i)_{i=1}^k$, then if $u \in H^1(\Omega, \mathbb{S}^1)$ then
$$\sum_{i=1}^k \mathrm{deg}(u_{\vert \Gamma_i})=0 .$$
In particular, for $a\in \Omega$ and $u\in H^{1}(\Omega\setminus\{a\},\mathbb{S}^1)$,  we set $\mathrm{deg}(u,a)=\mathrm{deg}(u_{\vert \partial B(a,r)})$ for $r>0$ small enough.
\end{theo}

\section{Proof of Theorem \ref{theoGL}}

\subsection{Upper and lower log estimates for the energy} 
\label{subenergyestimates}
Let $u_{\epsilon} \in \mathcal{A}$ be such that $E_{\eps}(u_{\eps}) = \inf_{u \in \mathcal{A}} E_{\eps}\left(u \right)$. Such a minimizer exists because the functional $E_{\eps}$ is coercive. Since $u_{\eps}$ comes from a minimization problem, by a test function argument, we have the following upper bound on the energy, see Theorem III.1 of \cite{BBH}. There exists $D_0>0$ such that, for any $\eps>0$,
\begin{equation} \label{upperboundBBH} E_{\eps}(u_{\eps}) \leq \pi \left\vert \ln \eps \right\vert + D_0 . \end{equation}

Then we want to prove the following key estimate,

\begin{equation}\label{eqboundedassumption} \frac{1}{4\eps^2}\int_{\Omega}\left(1-\vert u_{\eps}\vert^2\right)^2dx \leq C \end{equation} 
for some constant $C$ independent from $\eps$, depending only on $\Omega$ and $g$.

In their seminal work Bethuel, Brezis and Hélein proved this estimate using a Poho\v{z}aev identity, but in our case since the boundary data is only $H^\frac{1}{2}$ we can't apply this idea.

Our argument, inspired from \cite{dPF}, is based on an extension of the sequence of minimizers $u_{\eps}:\Omega \to \mathbb{R}^2$ to some bigger smooth domain $\Omega' \supset \Omega$ whose boundary data is smooth on $\partial \Omega'$.

\begin{lem}[extension lemma]
\label{extl}
Let $\Omega$ a domain with a quasicircle as boundary and $g\in H^{\frac{1}{2}}\left(\partial \Omega,\mathbb{S}^1\right)$ there is a smooth simply connected domain $\Omega'$ such that $\overline{\Omega} \subset \Omega'$ and  a function $\hat{u} : \Omega'\setminus \Omega \to \mathbb{R}^2$ such that
\begin{itemize}
\item $\hat{u}=g$ on $\partial \Omega$,
\item $\hat{u} \in \mathcal{C}^{\infty}\left( \overline{\Omega'} \setminus \overline{\Omega} \right)$,
\item $\hat{u} \in H^1\left( \Omega' \setminus \Omega \right)$ and whose energy depends only on $\Omega$ and $\Vert g\Vert_{H^\frac{1}{2}}$,
\item $ \left\vert \hat{u} \right\vert = 1 $ on  $\overline{\Omega'} \setminus \Omega$,
\item $\mathrm{deg}(\hat{u}_{\vert \partial \Omega'})=\mathrm{deg}(g)$.
\end{itemize}
\end{lem}

\begin{proof} 
Let $h :\D^c\rightarrow \Omega^c$ be a uniformization map, since $\Omega$ is a quasidisk, thanks to Theorem \ref{Extension}, $h$ admits a quasiconformal extension from $\C$ to itself, we still denote $h$.
We set $\tilde{g}:\mathbb{S}^1 \rightarrow \mathbb{S}^1$ by $\tilde{g}=g\circ h$, by conformal invariance, see Theorem \ref{H12}, $\tilde{g}$ is in $H^\frac{1}{2}(\mathbb{S}^1,\mathbb{S}^1)$. Let $\tilde{u} \in H^1(\D^c) \cap C^{\infty}(\D^c,\D)$ be the harmonic extension of $\tilde{g}$, finally we set $v=\tilde{u} \circ h^{-1}$.  Since $h$ and $\tilde{u}$ are smooth outside $\overline{\D}^c$ then $v$ is smooth on $\overline{\Omega}^c$. Since $h$ is conformal on $\overline{\D}^c$, the energy of $v$ is bounded and depends only on $\Vert \tilde{g}\Vert_{H^\frac{1}{2}}$, i.e on $\Omega$ and $g$ by Theorem \ref{H12} . Finally, thanks to Lemma \ref{Klem}, there exists $\delta>0$ such that $\frac{3}{4}\leq \vert \tilde{u} \vert \leq \frac{5}{4}$ on $\D_{1+\delta}\setminus \D$. Then $\hat{u}=\frac{v}{\vert v\vert }$ satisfies the desired properties on $\Omega'=g(\D_{1+\delta})$ since it has the same property than $v$ and  the last property is given by Theorem \ref{somdeg}.    
\end{proof}

Then, thanks to Lemma \ref{extl}, the function $\hat{u}_{\eps}:\Omega' \to \mathbb{R}^2$ defined as an extension of $u_{\eps}$ such that
$$\hat{u}_{\eps} = u_{\eps} \text{ in }\Omega \text{ and } \hat{u}_{\eps} = \hat{u} \text{ in } \Omega'\setminus\Omega$$ 
can be used as a test function for the classical Ginzburg-Landau problem on the smooth domain $\Omega'$, minimizing the energy
$$ E^{\Omega'}_{\eps}(v) = \frac{1}{2}\int_{\Omega'}\left\vert \nabla v \right\vert^2 + \frac{1}{4\eps^2} \int_{\Omega'}\left(1-\left\vert v \right\vert^2\right)^2$$
on the set of admissible functions 
$\mathcal{A}^{\Omega'} = \{v \in W^{1,2}\left(\Omega'\right) ; v = \hat{u} \text{ on } \partial \Omega'\}$. We recall that $\hat{u}$ is a smooth function on $\partial \Omega'$.
Since
\begin{equation*}  E_{2\eps}(u_{\eps}) = E^{\Omega'}_{2\eps}(\hat{u}_{\eps}) - \frac{1}{2}\int_{\Omega'\setminus\Omega} \left\vert \nabla \hat{u} \right\vert^2 \geq \inf_{v\in \mathcal{A}^{\Omega'}} E^{\Omega'}_{2\eps}(v) - \frac{1}{2}\int_{\Omega' \setminus\Omega} \left\vert \nabla \hat{u} \right\vert^2  
\end{equation*}
and by the classical lower bound on the energy, see estimate (1.9) in \cite{dPF}, we finaly get
\begin{equation} \label{eqlowerlogest} E_{2\eps}(u_{\eps}) \geq \pi\left\vert \ln\eps \right\vert - D_1 \end{equation}
for some constant $D_1$ depending only on $\Omega$ and $\tau$.

Now, using \eqref{upperboundBBH} and \eqref{eqlowerlogest} we obtain
$$ \frac{1}{4\eps^2} \int_{\Omega} \left(1-\left\vert u_{\eps} \right\vert^2\right)^2 = \frac{4}{3}\left( E_{\eps}(u_{\eps}) - E_{2\eps}(u_{\eps}) \right) \leq \frac{4}{3}\left( D_0 +D_1   \right).$$
Setting $C = \frac{4}{3}\left( D_0 +D_1 \right)$ ends the proof of \eqref{eqboundedassumption}.

\subsection{A priori estimates for sequences of minimizers} 
In this section we prove two fundamental estimate,  the first one, proposition \ref{propfarboundary}, asserts that $u_\eps$ can't be too small near the boundary, and the second one, proposition \label{prquantum}, is known as $\eta$-compactness. Their proof does not require that $u_\eps$ is a minimizer but they only need as assumption the uniform boundness of $u_\eps$. This boundness is an easy consequence of the maximum principle in the smooth case, but in our our weak setting we are not able to prove it without assuming that the $u_\eps$ is a minimizer. Which leads us to ask the following question: 
\smallskip

\textbf{Open question :} are there sequences of critical point of $E_\eps$ with fixed boundary data in $H^\frac{1}{2}(\partial \Omega, \mathbb{S}^1)$ which are not uniformly bounded ? 

\medskip

Any critical point $u_{\eps}$ of $E_\eps$ satisfies the Euler-Lagrange equation
\begin{equation} 
\label{eqeulerlagrange} 
\begin{cases} -\Delta u_{\eps} = \frac{1}{\eps^2}\left(1-\left\vert u_{\eps} \right\vert^2\right) u_{\eps} & \text{ in }\Omega \\
u_{\eps} =_{a.e} g & \text{ on } \partial{\Omega} \end{cases}\end{equation}
and we have those very first estimates on $u_{\eps}$, which are true as soon as $g \in H^\frac{1}{2}\left(\Gamma,\mathbb{S}^1\right)$. In fatc in \cite{FM} they assume that boundary is Lipschitz but the argument works as soon as we have a minimizer.

\begin{prop}[Lemma 3.2 of \cite{FM}]
\label{norm1}

If $u_\eps$ is a minimizer of $E_{\eps}$, then

\begin{equation} \label{eqmaxprinciple} \left\Vert u_{\eps} \right\Vert_{L^\infty\left(\Omega\right)} \leq 1 .
\end{equation}
\end{prop}
Notice that  \eqref{eqmaxprinciple} is true for any critical solution $u_{\eps}$ to $E_{\eps}$ if $\Omega$ is smooth enough to apply the maximum principle argument, see proposition 2 of \cite{BBH2}. The next proposition is a classical elliptic estimate true independently of the regularity of $\Omega$.

\begin{prop}[Lemma A.1 \cite{BBH2}] \label{propBBHgradientbound}
There is a constant $C_1>0$ such that for any $\eps>0$, and for any $x\in \Omega$,
\begin{equation} \label{eqgradientestimate} \left\vert \nabla u_{\eps} \right\vert(x) \leq C_1 \Vert u_\eps \Vert_\infty \left( \frac{1}{\eps} + \frac{1}{dist(x,\partial\Omega)}  \right). \end{equation}
In particular, if \eqref{eqmaxprinciple} holds, $\Vert u_\eps\Vert_\infty$ can be removed.
\end{prop}

In the following proposition, we prove that if $u_\eps$ is uniformly bounded then  $\left\vert u_{\eps}(x)\right\vert$ has to be far from $0$ as soon as $x$ is close to $\partial \Omega$ at the scale $\eps$. This hypothesis of uniform  boundness of $u_\eps$ is satisfies in the case of minimizers thanks to proposition \ref{norm1}.

\begin{prop} \label{propfarboundary} Let $u_\eps$ be solutions of \eqref{eqeulerlagrange} which are uniformly bounded with respect to $\eps$, then  there are $\eps_0,\eta_0>0$ depending only on $\sup_{\eps} \Vert u_\eps \Vert_\infty $, $\Omega$ and $g$ such that for any $0<\eps\leq \eps_0$, for any $x\in \Omega$ 

\begin{equation}\label{eqnormuclosetoboundary} 
 dist(x,\partial\Omega)\leq \eta_0 \eps  
 \Rightarrow \hspace{3mm} \left\vert u_{\eps}(x)\right\vert \geq \frac{1}{2}.  \end{equation}
\end{prop}

\begin{proof}  First of all, let us remind that thanks to our assumption and \eqref{eqboundedassumption}, there exists $C>0$ such that 
$$ \frac{1}{4\eps^2}\int_{\Omega}\left(1-\vert u_{\eps}\vert^2\right)^2dx \leq C  \text{ and } \Vert u_\eps\Vert_\infty \leq C. $$ 

We aim at finding $\eta_0$ and $\eps_0$ such that for any $0<\eps \leq \eps_0$ such that for any $y_\eps\in \partial\Omega$,
\begin{equation} \label{eqgoalueps} \inf_{x\in D(y_\eps,\eta_0\eps)\cap \Omega}\left\vert u_{\eps}(x)\right\vert \geq \frac{1}{2}. \end{equation}
Let $f:\mathbb{D}\to \Omega$ be a uniformization of $\Omega$. Since $\Gamma = \partial\Omega$ is a quasicircle, there is a $K$-quasiconformal extension $f:\mathbb{C}\to \mathbb{C}$ of $f$, still denoted $f$. We let $v_{\eps} = u_{\eps}\circ f : \mathbb{D}\to \mathbb{R}^2$ be the pullback of $u_{\eps}$ by $f$. We have by conformal invariance, and because $u_{\eps}\in H^{\frac{1}{2}}\left(\Gamma \right)$ that
$$ \begin{cases} -\Delta v_{\eps} = \left\vert f' \right\vert^2 \frac{1-\left\vert v_{\eps} \right\vert^2}{\eps^2} v_{\eps} & \text{ in } \mathbb{D} \\
v_{\eps}=g \circ f \in  H^{\frac{1}{2}}\left(\mathbb{S}^1\right) & \text{ on } \mathbb{S}^1 \end{cases}$$

Let $\lambda_0>0$ we shall fix later, we set $D_\eps =f^{-1}\left(\mathbb{D}_{\lambda_0\eps}(y_{\eps})\cap \Omega \right)$. We split $v_{\eps} = a_{\eps} + b_{\eps}+h$ so that
\begin{align*}
& \begin{cases} -\Delta a_{\eps} = \mathbf{1}_{D_\eps} \left\vert f' \right\vert^2 \frac{1-\left\vert v_{\eps} \right\vert^2}{\eps^2} v_{\eps} & \text{ in } \mathbb{D} \\
a_{\eps} = 0 & \text{ on } \mathbb{S}^1 \end{cases}\\
&\begin{cases} -\Delta b_{\eps} = \left(1-\mathbf{1}_{D_\eps}\right) \left\vert f' \right\vert^2 \frac{1-\left\vert v_{\eps} \right\vert^2}{\eps^2} v_{\eps} & \text{ in } \mathbb{D} \\
b_{\eps} = 0 & \text{ on } \mathbb{S}^1 \end{cases} \\
&\begin{cases} -\Delta h = 0 & \text{ in } \mathbb{D} \\
h = v_{\eps} = g \circ f & \text{ on } \mathbb{S}^1 \end{cases}
\end{align*}

\bigskip

\noindent \textbf{ Claim 1: For $\lambda_0$ small enough $\Vert a_\eps \Vert_\infty \leq \frac{1}{8}$.}
\smallskip

By Lemma \ref{LLP}, we have
\begin{equation} \label{eqLinfty+gradientesta} \left\Vert a_{\eps} \right\Vert_{L^\infty} + \left\Vert \nabla a_{\eps} \right\Vert_{L^2} \leq C_0 \left( \left\Vert \left(1-\left\vert z \right\vert^2\right)^2 \Delta a_{\eps} \right\Vert_{L^{\infty}} + \left\Vert \Delta a_{\eps} \right\Vert_{L^1} \right). \end{equation}
We aim at estimating the right-hand terms. By Koebe distortion theorem, see Corollary 1.4 \cite{Pom}, we have that 
\begin{equation} \label{KoebeSchwartz}  \frac{dist(f(z),\partial \Omega)}{1-\left\vert z \right\vert^2}  \leq \left\vert f'(z) \right\vert \leq 4 \frac{dist(f(z),\partial \Omega)}{1-\left\vert z\right\vert^2}. \end{equation}
In particular, if $z \in D_\eps $, 
\begin{equation} \label{eqdistboundaryLinftyest}  \left\vert f'(z) \right\vert^2 \frac{1-\left\vert v_{\eps}(z) \right\vert^2}{\eps^2} \left\vert v_{\eps}(z) \right\vert \leq 4 C \frac{ \lambda_0^2}{(1-\left\vert z \right\vert^2)^2}. \end{equation}
We also have that, by change of variables,
\begin{equation} \label{eqL1est} \begin{split} 
\int_{D_\eps} \left\vert f' \right\vert^2 \frac{1-\left\vert v_{\eps} \right\vert^2}{\eps^2} \left\vert v_{\eps} \right\vert \, dz  & = \int_{f(D_\eps)} \frac{1-\left\vert u_{\eps} \right\vert^2}{\eps^2} \left\vert u_{\eps} \right\vert \, dz \\
& \leq \left(\int_{\Omega}\frac{\left(1-\left\vert u_{\eps} \right\vert^2\right)^2}{\eps^2} \, dz\right)^{\frac{1}{2}} \left( C\frac{\left\vert f(D_\eps) \right\vert}{\eps^2}  \right)^{\frac{1}{2}} \\
& \leq 2C \sqrt{2\pi} \lambda_0 .
\end{split} 
\end{equation}
Combining \eqref{eqLinfty+gradientesta}, \eqref{eqdistboundaryLinftyest} and \eqref{eqL1est}, we obtain
\begin{equation} \label{eqLinfty+gradientesta2} \left\Vert a_{\eps} \right\Vert_{L^\infty} + \left\Vert \nabla a_{\eps} \right\Vert_{L^2} \leq C \left( 4\lambda_0^2 + 2 \sqrt{2\pi} \lambda_0 \right). \end{equation}
Choosing $\lambda_0$ small enough such that $C \left( 4\lambda_0^2 + 2 \sqrt{2\pi} \lambda_0 \right)\leq \frac{1}{8}$, we obtain that
\begin{equation} \label{eqLinftysmall} \left\Vert a_{\eps} \right\Vert_{L^\infty} \leq \frac{1}{8}. \end{equation}

\noindent \textbf{ Claim 2: There exists $r_0>0$ such that  $\Vert b_\eps \Vert_{L^\infty(f^{-1}\left(\mathbb{D}_{r_0 \lambda_0\eps}(y_{\eps})\cap \Omega \right))} \leq \frac{1}{8}$.}
\smallskip

Let 
$$\delta_\eps = \frac{\mathrm{dist} (f^{-1}(y_\eps),\partial (f^{-1}(\D_{\lambda_0 \eps}(y_\eps) ))) }{2},$$
 then

\begin{equation}  \mathbb{D}_{ \delta_\eps}(f^{-1}(y_{\eps}) ) \subset f^{-1}\left(\mathbb{D}_{\lambda_0 \eps}(y_{\eps}) \right)
\end{equation}
We set $\tilde{b}_{\eps}(z) = b_{\eps}(f^{-1}(y_{\eps})+ \delta_\eps z)$. Then, let $B_\eps=\frac{\mathbb{D}-f^{-1}(y_{\eps})}{\delta_\eps}\cap \mathbb{D}$, we have that
$$ \begin{cases} \Delta \tilde{b}_{\eps} = 0 & \text{ in } B_\eps  \\
\tilde{b}_{\eps} = 0 & \text{ on } \partial B_\eps \cap \mathbb{D}. \end{cases}$$
By standard elliptic estimates, there is a constant $D>0$ independent from $\eps>0$ such that
\begin{equation} \label{eqellipticonb} \left\Vert \nabla\tilde{b}_{\eps} \right\Vert_{L^\infty (B_\eps \cap \D_{\frac{1}{2}})} \leq D \left\Vert \tilde{b}_{\eps} \right\Vert_{L^\infty\left( B_\eps \right)} . \end{equation}
We also know by $\eqref{eqLinftysmall}$, the fact that $v_\eps$ 
is bounded and that $h \in \D$ as harmonic extension of a $\mathbb{S}^1$ valued function, that
\begin{equation}
\label{bbond}
\left\Vert \tilde{b}_{\eps} \right\Vert_{L^\infty\left( B_\eps \right)} \leq \left\Vert b_{\eps} \right\Vert_{L^\infty} \leq \left\Vert v_{\eps} \right\Vert_{L^\infty} + \left\Vert a_{\eps} \right\Vert_{L^\infty} + \left\Vert h \right\Vert_{L^\infty} \leq C + \frac{1}{8} + 1 \leq C + 2.
\end{equation}
In particular, since we have that $\tilde{b}_{\eps} = 0$ on $\partial B_\eps \cap \mathbb{D}$, thanks to \eqref{eqellipticonb} and \eqref{bbond}, there is a radius $0<\rho_0<\frac{1}{2}$ such that
\begin{equation} \label{eqsmallbepstilde} \left\Vert b_{\eps} \right\Vert_{L^\infty\left( \mathbb{D} \cap \mathbb{D}_{\rho_0 \delta_\eps}(f^{-1}(y_{\eps})) \right)}  = \left\Vert \tilde{b}_{\eps} \right\Vert_{L^\infty\left( B_\eps \cap \mathbb{D}_{\rho_0} \right)} \leq \frac{1}{8}. \end{equation}
Finally since $f^{-1}$ is also quasiconformal, see  Corollary 4.5.10 of \cite{Hu}, thanks to Theorem \ref{KQCLQS}, for every $x\in \partial\D_{\lambda_0 \eps} (y_\eps)$ and $y\in \D_{r_0\lambda_0 \eps} (y_\eps)$ we have
$$ \frac{\vert f^{-1}(y)-f^{-1}(y_\eps)\vert }{\vert f^{-1}(x)-f^{-1}(y_\eps)\vert } \leq \eta(r_0)$$
which gives
$$ \vert f^{-1}(y)-f^{-1}(y_\eps)\vert  \leq \eta(r_0) (2\eta(1) \delta_\eps)$$
hence there is a constant $r_0>1$, smaller than $1$, depending only on the quasiconformal constant $K$ of $f^{-1}$, that is to say only on $\Omega$, such that 
\begin{equation}
\label{distorsionlemma2}
f^{-1}\left(\mathbb{D}_{r_0 \lambda_0 \eps}(y_{\eps})\cap \Omega \right) \subset  \mathbb{D}_{\rho_0 \delta_{\eps}}(f^{-1}(y_{\eps}))
\end{equation}
hence \eqref{eqsmallbepstilde} and \eqref{distorsionlemma2} give that
\begin{equation} \label{eqsmallbeps} \left\Vert \tilde{b}_{\eps} \right\Vert_{L^\infty\left( f^{-1}\left(\mathbb{D}_{r_0 \lambda_0 \eps}(y_{\eps})\cap \Omega \right) \right)}  \leq \frac{1}{8},
\end{equation}
which achieves the proof of claim 2.

\medskip

Now, we set $\eta_0 = r_0 \lambda_0$. To complete the proof it sufficies to proves that $\Vert h \Vert_{L^{\infty}(f^{-1}\left(\mathbb{D}_{\eta_0 \eps}(y_{\eps})\cap \Omega\right)} \geq \frac{3}{4}$. Notice that the $H^{\frac{1}{2}}\left(\mathbb{S}^1\right)$ norm of $h =v_\eps =u_\eps\circ f= g \circ f$ is independent of $\eps$, hence the result follows because $h$ is uniformly proper by Lemma \ref{Klem}. 
\end{proof}

The following proposition is an easy consequence of Proposition \ref{propBBHgradientbound} and Proposition \ref{propfarboundary}.

\begin{prop} \label{prquantum} Let $u_\eps$ be solutions of \eqref{eqeulerlagrange} which are uniformly bounded with respect to $\eps$, then  there are $\eps_0,\eta_0,\delta_0 >0$ depending only on $\sup_{\eps} \Vert u_\eps \Vert_\infty $, $\Omega$ and $g$ such that for any $0<\eps\leq \eps_0$, for any $x\in \Omega$ %
\begin{equation}\label{eqquantum}
\begin{split} 
\left\vert u_{\eps}(x)\right\vert \leq \frac{1}{2} 
\Rightarrow  \frac{1}{4\eps^2}\int_{\mathbb{D}_{\eta_0\eps}(x)}\left(1-\vert u_{\eps}\vert^2\right)^2dx \geq \delta_0. 
\end{split}
\end{equation}
\end{prop}

\subsection{Weak convergence to $u_*$}

From now to the end of the subsection, we follow the lines of part 3 in \cite{Struwe94}. First, the following claim is a very simple consequence of \eqref{eqboundedassumption}, Lemma \ref{norm1} and Proposition \ref{prquantum}.

\begin{cl} \label{cluniformlyfinitedisjointball} There are $\eta_0>0$ and  $J_0\in \mathbb{N}$ depending only on $\Omega, g$ such that for any disjoint collection of balls $\left(\mathbb{D}_{\eta_0\eps}(x_j^{\eps})\right)_{1\leq j\leq J}$ such that $\left\vert u_{\eps}(x_j^{\eps}) \right\vert \leq \frac{1}{2}$, then $J\leq J_0$.
\end{cl}

By a classical Vitali's covering lemma, we can find a family of disjoint balls $\left(\mathbb{D}_{\eta_0\eps}(x_j^{\eps})\right)_{1\leq j\leq J_{\eps}}$ such that $\left\vert u_{\eps}(x_j^{\eps}) \right\vert \leq \frac{1}{2}$ for any $1\leq j\leq J_{\eps}$ and
$$ \left\{ x\in \Omega ; \left\vert u_{\eps}(x) \right\vert \leq \frac{1}{2} \right\} \subset \bigcup_{1\leq j \leq J_{\eps}} \mathbb{D}_{5\eta_0\eps}(x_j^{\eps}).$$
By Claim \ref{cluniformlyfinitedisjointball}, $J_{\eps} \leq J_0$ and for some radius $r>0$, we denote by
$$ {\Omega'}_{\eps,r} = \Omega' \setminus \bigcup_{1\leq j \leq J_{\eps}} \mathbb{D}_{r}(x_j^{\eps}) \text{ and } \Omega_{\eps,r} = {\Omega'}_{\eps,r} \cap \Omega$$
where we recall that $\Omega'$ is the extension of $\Omega$ given by Lemma \ref{extl}. ${\Omega'}_{\eps,r} $ is defined for technical reasons since we have the property $\bigcup_{1\leq j \leq J_{\eps}} \mathbb{D}_{r}(x_j^{\eps}) \subset \Omega'$ for $r$ small enough.

Notice that by definition, $\left\vert u_{\eps}(x) \right\vert \geq \frac{1}{2}$ for $x\in {\Omega}_{\eps,r}$ if $r\geq 5\nu_0\eps$ and that it is also true for the extension $\hat{u}_{\eps}$ given by Lemma \ref{extl} then: 
$$ \left\vert \hat{u}_{\eps}(x) \right\vert \geq \frac{1}{2} \text{ for } x\in {\Omega'}_{\eps,r} \text{ if } r\geq 5\nu_0\eps $$

\medskip

Let us remind the following general proposition orginally due to Brezis, Merle and Rivière \cite{BMR}.
\begin{prop}[\cite{Struwe94} Proposition 3.4 and \cite{Struwe94bis}] \label{propBMRStruwe}

Let $0<r_0 \leq r_1 \leq r_2$ and $C>0$. Then for any $ v\in H^1\left( A_{r_0,r_2} \right)$ such that $\frac{1}{2} \leq \left\vert v \right\vert \leq 1$ a.e on $A_{r_0,r_1}$ and such that $v$ has degree $d$ on every circle $S_r$ for $r_1 \leq r\leq r_2$, and such that
$$ \frac{1}{r_1^2} \int_{A_{r_0,r_2}} \left(1-\vert v \vert^2 \right)^2 \leq C \text{ and } \int_{A_{r_0,r_2}} \left\vert \nabla v \right\vert^2 \leq C\vert\ln r_0\vert ,$$
then for any $r_1 \leq r \leq r_2$,
$$ \int_{A_{r_1,r} } \left\vert \nabla v \right\vert^2 \geq 2\pi d^2 \left(\ln\left(\frac{r}{r_1}\right)- D\right) - D$$
where $D$ only depends on $C$.
\end{prop}

The next proposition is a consequence of \eqref{upperboundBBH}, Proposition \ref{propBMRStruwe} applied on the extended function $\hat{u}_{\eps}$ on an accurate disjoint union of annuli centered at points in the set $\left(x_{j}^{\eps}\right)_{1\leq j \leq J_{\eps}}$ almost filling $\left(\Omega'\right)_{\eps,r}$ and the fact that the degree of $\hat{u}_{\eps}$ is $1$ on $\partial\Omega'$.

\begin{prop}[\cite{Struwe94} Proposition 3.3] \label{propenergyeststruwe}
There are constants $\kappa>0$, $C$ depending only on $\Omega$ and $g$ such that for any $\eps < r < \delta$,
\begin{equation} \label{eqenergyestimatefarbaddisks}  \int_{\Omega'_{\eps,r}} \left\vert \nabla u_{\eps} \right\vert^2 \leq  2\pi \vert \ln r \vert  + C \end{equation}
\end{prop}

\begin{cl}[Remark that follows Proposition 3.3 of \cite{Struwe94}] \label{propconvstruwe}
Let $\eps_{k} \to 0$ be a sequence, 

Then there is a subsequence $\eps_{k_l}$ such that $\left( J_{\eps_{k_l}} \right)_{l\in \mathbb{N}}$ is constant 
and such that for any $1\leq j \leq J:= J_{\eps_{k_l}}$ there is $a_j \in \overline{\Omega}$ such that $x_j^{\eps_{k_l}} \to a_j$ as $l\to +\infty$. We denote
$$ S = \{a_j ; 1\leq j \leq J\} $$

Moreover there is $a \in S$ and $u_{\star} \in H^1_{loc} (\overline{\Omega}\setminus\{a\}) \cap \mathcal{C}^{\infty}\left(\Omega\setminus \{a\} \right)$ such that $u_{\eps_{k_l}} \to u_{\star}$ in $\mathcal{C}_{loc}^{\infty}\left(\Omega\setminus S \right) $  as $l\to +\infty$. Moreover, we have that
\begin{equation} \label{eqenergyestimatefara}  \int_{\Omega\setminus \mathbb{D}_r(a)} \left\vert \nabla u_{\star} \right\vert^2 \leq  2\pi \ln\left(\frac{1}{r}\right) + \kappa . \end{equation}
\begin{equation} \label{eqenergyestimatefaraannulus}  \int_{ \mathbb{D}_{r_2}(a) \setminus \mathbb{D}_{r_1}(a)} \left\vert \nabla u_{\star} \right\vert^2 \leq  2\pi \ln\left(\frac{r_2}{r_1}\right) + \kappa' . \end{equation}
and that $u_{\star}$ is a harmonic map into $\mathbb{S}^1$.
\end{cl}

Notice that combining proposition \ref{propBMRStruwe} and Claim \ref{propconvstruwe}, $a$ is a degree one singularity for $u_{\star}$ and the other points of $S\setminus\{a\}$ are degree zero singularities. In the next subsections, we aim at proving that $a\in \Omega$ is an interior point, that $S\setminus\{a\} = \emptyset$ and that $u_{\eps_{k_l}} \to u_{\star}$ in $H^{1}_{loc}\left(\overline\Omega\setminus \{a\} \right) $ as $l\to +\infty$.
These properties will complete the proof of (i)-(v) of Theorem \ref{theoGL}. The last point of Theorem \ref{theoGL} is a direct consequence of Proposition 7.1 \cite{FM}.

\subsection{The degree one singularity is an interior point of the domain} \label{secinteriorsingularity}

We aim at proving that $a\in \Omega$. Let us assume by contradiction that $a \in \partial \Omega$ and without loss of generality that $a=0$.

As in the proof of of the convergence of $u_\eps$, $u_*$ is extended  to a smooth domain $\Omega'$ thanks to Lemma \ref{extl}. In order to light the notation, we will denote its extension $u_*$. Moreover we can also apply Lemma \ref{extl} to $\Omega^c$ in order get the following claim
\begin{cl}
\label{extension3}
Let $\Gamma$ a quasicircle and $g\in H^\frac{1}{2}(\Gamma,\mathbb{S}^1)$ then there exists a smooth neighborhood $\Omega''$ of $\Gamma$ diffeomorphic to an annulus and $u_0\in H^{1}(\Omega'',\mathbb{S}^1)$ such that $u=g$ on $\Gamma$.
\end{cl}

We can assume that $\Omega''\setminus \Omega= \Omega'\setminus\Omega$. Let $t>0$ small enough such that $\D_t \subset \Omega''$.

We first proof that on $S_t \cap \Omega'$ there is always an arc with angle uniformly bounded from below.

\begin{cl}
\label{clarc}
there exists $\delta>0$ such that $\D_t \cap \Omega'$ contains a connected arc of arc length  greater than $\delta t$, we denote it $A_t'$.
\end{cl}

\begin{proof}
\begin{figure}[h]
\includegraphics[page=3]{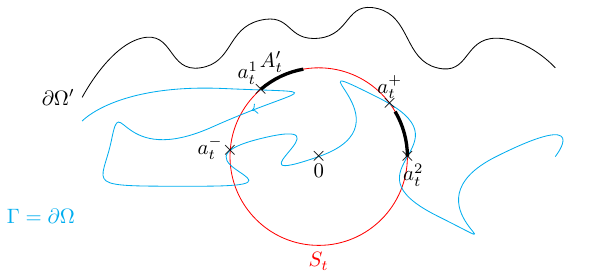}
\caption{Uniform control on $S_t\cap \Omega'$}

\end{figure}

We choose some orientation on $\Gamma$, let $a^-_t$ the last point of $S_t \cap \Gamma$ before $0$ and $a_t^+$ the first point of $S_t\cap \Gamma$ after $0$. We denote by $\Gamma_-$  the open arc of $\Gamma$ from $a^+_t$ to $0$. Thanks to \eqref{eqs}, there is no point of $\Gamma^-$  in $B(a^+_t,\frac{t}{C})\cap \Gamma$. let $A_t^1$ and $A_t^2$ the two open connected component of $S_t\setminus\{ a_t^-,a_t^+\}$, and
let $a_t^{1}$ (resp. $a_t^{2}$) the closest point  to $a_t^+$ of $\Gamma_-\cap A_t^1$ (resp. $\Gamma_-\cap A_t^2$). By the previous remark, the distance of those points to $a_t^+$ is bigger than $\frac{t}{2C}$, Hence there is the two open arc starting at the $a^i_t$ of arc length $\frac{t}{2C}$ which contains no point of $\Gamma$, one of this two is included in $\Omega'$.
\end{proof}

By \eqref{eqdeg} and the fact $0$ is a degree one singularity, we have

\begin{align*}
1=\mathrm{deg}(u_*,0)&=\frac{1}{2\pi} \int_{S_t} u_* \wedge{u_*}_\theta \, d\theta \\
&=\frac{1}{2\pi} \int_{A_t'} u_* \wedge{u_*}_\theta \, d\theta +\frac{1}{2\pi} \int_{\Gamma\setminus A_t'} u_* \wedge {u_*}_\theta \, d\theta .
\end{align*}
Then, we have

$$\frac{1}{2\pi} \left \vert \int_{A_t'} u_* \wedge{u_*}_\theta\, d\theta \right\vert \leq \frac{1}{2\pi} \Vert u_0 \Vert_{H^\frac{1}{2}(S_t)}^2 \leq C \Vert u_0 \Vert_{H^1(A_t)}^2,$$
where $A_t = \D_{2t}\setminus \D_\frac{t}{2}$, and

$$\frac{1}{2\pi} \left\vert \int_{\Gamma\setminus A_t'} u_* \wedge{u_*}_\theta \, d\theta \right\vert \leq
\frac{1}{2\pi} \int_{\Gamma\setminus A_t'} \vert {u_*}_\theta\vert \, d\theta  \leq
\frac{1}{2\pi} \left( \int_{\Gamma\setminus A_t'} \vert {u_*}_\theta\vert^2  \, d\theta \right)^\frac{1}{2} 
\sqrt{(2\pi-\delta)t}
$$

Here we use the fact that the length of $A_t^t$ is at least $\delta t$. Then, we have

$$1- C \Vert \nabla u_0 \Vert^2_{L^2(A_t)} \leq \frac{\sqrt{(2\pi-\delta)t}}{2\pi} \left( \int_{S_t} \vert {u_*}_\theta\vert^2 \, d\theta \right)^\frac{1}{2},$$
hence by squaring this identity
$$\frac{4\pi^2}{(2\pi-\delta)t} (1- 2C \Vert \nabla u_0 \Vert^2_{L^2(A_t)}) \leq \int_{S_t} \vert {u_*}_\theta\vert^2  \, d\theta ,$$

But, thanks to Fubini Theorem, we have

\begin{align*}
\int_{r_1}^{r_2} \frac{\Vert \nabla u_0 \Vert_{L^2(A_t)}^2}{t} \,dt &= \int_{r_1}^{r_2} \frac{1}{t} \int_{A_{\frac{r_1}{2},2r_2}} \vert \nabla u_0 \vert^2 \pmb{1}_{\frac{t}{2}\leq \vert x\vert \leq 2t} \, dx\,dt \\
&= \int_{A_{\frac{r_1}{2},2 r_2}} \vert \nabla u_0 \vert^2   \int_{r_1}^{r_2} \frac{1}{t}  \,dt\, dx\\
&=\ln(4) \int_{A_{\frac{r_1}{2},2r_2}} \vert \nabla u_0 \vert^2 \, dx . 
\end{align*}
Finally we get

$$\frac{4\pi^2}{(2\pi-\delta)t} \ln \left( \frac{r_2}{r_1} \right) - C \Vert \nabla u_0 \Vert_{L^2\left(A_{\frac{r_1}{2},2r_2}\right)} \leq \int_{A_{r_1,r_2}} \vert \nabla u_* \vert^2 \,dx.$$
But thanks to \eqref{eqenergyestimatefaraannulus}, we know that

$$ \int_{A_{r_1,r_2}} \vert \nabla u_* \vert^2 \,dx \leq 2\pi \ln \left( \frac{r_2}{r_1} \right) + C,$$
which leads to a contradiction taking $\frac{r_2}{r_1}$ small enough.
  
\subsection{Final convergence arguments}

\begin{prop}
\label{lastprop} With the notations of Claim \ref{propconvstruwe}, we have that $u_{\eps_{k_l}} \to u_{\star}$ in $H^{1}_{loc}\left(\overline\Omega\setminus \{a\} \right) $ and for any compact set $K\Subset \overline{\Omega}\setminus\{a\}$,
\begin{equation} \label{eqnoquantumoutsidea} \frac{1}{4\eps_{k_l}^2} \int_{K}\left(1-\vert u_{\eps_{k_l}} \vert^2 \right)^2 \to 0 \text{ as } l \to +\infty \end{equation}
\end{prop}

\begin{proof}
Let $r_0>0$ small enough be such that $\mathbb{D}_{r_0}(a) \cap \left(\partial\Omega \cup S \right)= \emptyset$.
Since $u_{\eps_{k_l}}$ is bounded in $H^1\left( \Omega \setminus \mathbb{D}_{r_0}(a) \right)$, $u_{\eps_{k_l}}$ weakly converges to some function $u_{\star} \in H^1\left( \Omega \setminus \mathbb{D}_{r_0}(a) \right)$. Moreover, $u_{\star}$ is a harmonic map into $\mathbb{S}^1$ and $u_{\star}$ satisfies
\begin{equation}\label{eqenergyustarweaklimit} \int_{\Omega \setminus \mathbb{D}_{r_0}(a)} \left\vert \nabla u_{\star} \right\vert^2 \leq \liminf_{l \to +\infty}  \int_{\Omega \setminus \mathbb{D}_{r_0}(a)} \left\vert \nabla u_{\eps_{k_l}} \right\vert^2. \end{equation}

From now to the end of the proof of the proposition we denote $\eps:= \eps_{k_l}$.

Let $w_{\eps}$ be the solution the following minimization problem
$$ \frac{1}{2}\int_{\Omega \setminus \mathbb{D}_{r_0}(a)} \left\vert \nabla w_{\eps} \right\vert^2 = \inf_{u\in \mathcal{A}_{\eps}} \frac{1}{2}\int_{\Omega \setminus \mathbb{D}_{r_0}(a)} \left\vert \nabla u \right\vert^2$$
on the admissible set
$$ \mathcal{A}_{\eps} = \left\{ u\in H^1\left( \overline{\Omega} \setminus \mathbb{D}_{r_0}(a) , \mathbb{S}^1\right) ; u = \frac{u_{\eps}}{\left\vert u_{\eps} \right\vert} \text{ on } \partial \left(\Omega \setminus \mathbb{D}_{r_0}(a)\right) \right\} .$$

As explained in Claim \ref{extension3}, we can find $u^0_\eps$ is defined a tubular neighborhood $\Omega''$ of $\partial \Omega$ which is in $C^\infty (\Omega''\setminus,\mathbb{S}^1) \cap H^1 (\Omega'')$  and such that $u_\eps^0=u_\eps$ on $\partial \Omega$. Up to reduce $r_0$, we can assume that $\D(a,r_0)\cap \Omega''=\emptyset$. Then we set $\omega =\Omega\setminus \Omega''$. Let $\Psi: [a,b]\times \mathbb{S}^1 \to \overline{\omega} \setminus \mathbb{D}_{r_0}(a)$ be a smooth diffeomorphism such that $\Psi(\{a\}\times \mathbb{S}^1) = \partial\omega$ and $\Psi(\{b\}\times \mathbb{S}^1) = S(a,r_0)$. For instance, $\Psi$ can be build with a uniformization map from an annulus to $\overline{\omega} \setminus \mathbb{D}_{r_0}(a)$.
We set $a_{\eps} = u^0_\eps\circ \Psi_{\vert \{a\} \times \mathbb{S}^1}$ and $b_{\eps} = u_{\eps}\circ \Psi_{\vert \{b\}\times \mathbb{S}^1}$. Since $deg(a_{\eps}) = deg(b_{\eps})$, there is a smooth homotopy $H_{\eps}: [a,b]\times \mathbb{S}^1 \to \mathbb{S}^1$ such that $H_{\eps}(a,.) = a_{\eps}$ and $H_{\eps}(b,.) = b_{\eps}$ and the map 
\begin{equation*}
v_{\eps} := \begin{cases} H_{\eps}\circ \Psi^{-1} \text{ in } \omega\setminus  \mathbb{D}_{r_0}(a) \\ 
u_\eps^0 \text{ in } \Omega \setminus \omega
\end{cases}
\end{equation*}
satisfies $v_{\eps} \in \mathcal{A}_{\eps}$.

It is clear that since $w_{\eps} \to u_{\star} $ in $H^{\frac{1}{2}} \left( \partial \left(\Omega\setminus  \mathbb{D}_{r_0}(a) \right) \right)$ as $l\to +\infty$, $w_{\eps_{k_l}}$ converges in $H^1\left(\Omega\setminus  \mathbb{D}_{r_0}(a) \right)$ to some function $w_{\star}$ such that $w_{\star} = u_{\star}$ on $\partial \left(\Omega\setminus  \mathbb{D}_{r_0}(a) \right)$ and solution of
$$ \frac{1}{2}\int_{\Omega \setminus \mathbb{D}_{r_0}(a)} \left\vert \nabla w_{\star} \right\vert^2 = \inf_{u\in \mathcal{A}_{\star}} \frac{1}{2}\int_{\Omega \setminus \mathbb{D}_{r_0}(a)} \left\vert \nabla u \right\vert^2$$
on the admissible set
$$ \mathcal{A}_{\star} = \{ u\in H^1\left( \overline{\Omega} \setminus \mathbb{D}_{r_0}(a) , \mathbb{S}^1\right) ; u = u_{\star} \text{ on } \partial \left(\Omega \setminus \mathbb{D}_{r_0}(a)\right) \} .$$
In particular $u_{\star}\in \mathcal{A}_{\star}$ and
\begin{equation}\label{eqtestwstarustar} \int_{\Omega \setminus \mathbb{D}_{r_0}(a)} \left\vert \nabla w_{\star} \right\vert^2 \leq \int_{\Omega \setminus \mathbb{D}_{r_0}(a)} \left\vert \nabla u_{\star} \right\vert^2 \end{equation}
We let $f_{\eps}$ be the solution of
\begin{equation*}
\begin{cases}
\Delta f_{\eps} = 0 & \text{ in } \Omega \setminus \mathbb{D}_{r_0}(a) \\
f_{\eps} = 1 & \text{ on } \partial\Omega \\
f_{\eps} = \vert u_{\eps} \vert & \text{ on } S(a,r_0)
\end{cases}
\end{equation*}
Testing the function
\begin{equation*}
h_{\eps}:= \begin{cases} f_{\eps} w_{\eps} \text{ in } \Omega\setminus  \mathbb{D}_{r_0}(a) \\ 
u_{\eps} \text{ in }  \mathbb{D}_{r_0}(a)
\end{cases}
\end{equation*}
for the Ginzburg-Landau energy, we have that
$$ E_{\eps}(u_{\eps}) \leq E_{\eps}(h_{\eps}) .$$
so that since $h_{\eps} = u_{\eps}$ in $ \mathbb{D}_{r_0}(a)$ and  $\vert h_{\eps} \vert = f_{\eps}$ in $\Omega\setminus \mathbb{D}_{r_0}(a)$,
\begin{equation} \label{eqestimatebywf}
\begin{split}
 \frac{1}{2} \int_{\Omega\setminus \mathbb{D}_{r_0}(a)}& \left\vert \nabla u_{\eps} \right\vert^2 + \frac{1}{4\eps^2} \int_{\Omega\setminus \mathbb{D}_{r_0}(a)} \left( 1- \left\vert u_{\eps} \right\vert^2 \right)^2 \\
 & \leq  \frac{1}{2} \int_{\Omega\setminus \mathbb{D}_{r_0}(a)} f_{\eps}^2 \left\vert \nabla w_{\eps} \right\vert^2 +  \frac{1}{2} \int_{\Omega\setminus \mathbb{D}_{r_0}(a)}  \left\vert \nabla f_{\eps} \right\vert^2 + \frac{1}{4\eps^2} \int_{\Omega\setminus \mathbb{D}_{r_0}(a)} \left( 1- f_{\eps}^2 \right)^2 
 \end{split}
 \end{equation}

We set $\varphi_{\eps} = \frac{1-f_{\eps}}{\eps^2}$. Knowing that $f_{\eps}\leq 1$ by the maximum principle, we then easily obtain that 
\begin{equation} \label{eqestimatebyvarphi}
\begin{split}
 \frac{1}{2} \int_{\Omega\setminus \mathbb{D}_{r_0}(a)} & f_{\eps}^2 \left\vert \nabla w_{\eps} \right\vert^2 +  \frac{1}{2} \int_{\Omega\setminus \mathbb{D}_{r_0}(a)}  \left\vert \nabla f_{\eps} \right\vert^2 + \frac{1}{4\eps^2} \int_{\Omega\setminus \mathbb{D}_{r_0}(a)} \left( 1- f_{\eps}^2 \right)^2 \\
  \leq & \frac{1}{2} \int_{\Omega\setminus \mathbb{D}_{r_0}(a)} \left\vert \nabla w_{\eps} \right\vert^2 + \eps^2 \left(   \frac{\eps^2}{2} \int_{\Omega\setminus \mathbb{D}_{r_0}(a)}  \left\vert \nabla \varphi_{\eps} \right\vert^2 + \int_{\Omega\setminus \mathbb{D}_{r_0}(a)}  \varphi_{\eps}^2 \right)
 \end{split}
 \end{equation}

Now, we aim at proving that $\varphi_{\eps} $ is a non negative uniformly bounded function in $H^1 \cap L^{\infty}(\Omega\setminus  \mathbb{D}_{r_0}(a))$. Indeed, $\varphi_{\eps}$ is a harmonic function, uniformly bounded in $\mathcal{C}^\infty(S(a,r_0))$ since we know by Claim \ref{propconvstruwe} that
$$ \frac{ 1-\vert u_{\eps} \vert^2}{\eps^2} \to \left\vert \nabla u_{\star} \right\vert^2 \text{ on } \mathcal{C}^\infty(S(a,r_0)).$$
Setting $\tilde{\varphi}_{\eps} = \varphi_{\eps}\circ f$, where $f:\mathbb{D}\to \Omega$ is a uniformization map, we have that
\begin{equation*}
\begin{cases}
\Delta \tilde{\varphi}_{\eps} = 0 & \text{ in } \D \setminus f^{-1}(\mathbb{D}_{r_0}(a)) \\
\tilde{\varphi}_{\eps} = 0 & \text{ on } \mathbb{S}^1 \\
\tilde{\varphi}_{\eps} = \frac{1-\vert u_{\eps}\circ f \vert }{\eps^2}  & \text{ on } f^{-1}\left(S(a,r_0)\right)
\end{cases}
\end{equation*}
and by standard elliptic regularity, $\tilde{\varphi}_{\eps}$ is uniformly bounded. Coming back to $\varphi_{\eps} = \tilde{\varphi}_{\eps} \circ f^{-1} $, using the conformal invariance of the energy, we obtain the expected result.
 
 Now, using \eqref{eqestimatebywf}, \eqref{eqestimatebyvarphi} and the uniform boundedness of $\varphi_{\eps}$ in $H^1\cap L^{\infty}$, we obtain 
 \begin{equation} \label{eqestimatebyw}
\begin{split}
 \frac{1}{2} \int_{\Omega\setminus \mathbb{D}_{r_0}(a)}& \left\vert \nabla u_{\eps} \right\vert^2 + \frac{1}{4\eps^2} \int_{\Omega\setminus \mathbb{D}_{r_0}(a)} \left( 1- \left\vert u_{\eps} \right\vert^2 \right)^2 \\
 & \leq  \frac{1}{2} \int_{\Omega\setminus \mathbb{D}_{r_0}(a)}  \left\vert \nabla w_{\eps} \right\vert^2 + O(\eps^2)
 \end{split}
 \end{equation}
Passing to the limit on \eqref{eqestimatebyw} and using \eqref{eqtestwstarustar} and \eqref{eqenergyustarweaklimit}, we obtain
\begin{equation*} 
\begin{split}
\frac{1}{2}\int_{\Omega\setminus \mathbb{D}_{r_0}(a)} \left\vert \nabla u_{\star} \right\vert^2 \leq & \liminf_{\eps \to 0}  \int_{\Omega \setminus \mathbb{D}_{r_0}(a)} \left\vert \nabla u_{\eps} \right\vert^2 \\
 \leq & \limsup_{\eps \to 0}\left(  \int_{\Omega \setminus \mathbb{D}_{r_0}(a)} \left\vert \nabla u_{\eps} \right\vert^2 + \frac{1}{4\eps^2} \int_{\Omega \setminus \mathbb{D}_{r_0}(a)} \left( 1- \left\vert u_{\eps} \right\vert^2 \right)^2  \right) \\
  \leq & \frac{1}{2}\int_{\Omega\setminus \mathbb{D}_{r_0}(a)} \left\vert \nabla w_{\star} \right\vert^2 \leq \frac{1}{2}\int_{\Omega\setminus \mathbb{D}_{r_0}(a)} \left\vert \nabla u_{\star} \right\vert^2 
\end{split}
\end{equation*}
so that choosing $r_0$ as small as necessary, we obtain strong $H^1_{loc}\left(\Omega \setminus \{a\}\right)$ convergence of $u_{\eps}$ and \eqref{eqnoquantumoutsidea}.
\end{proof}

We obtain from  Proposition \ref{lastprop} and proposition \ref{prquantum} that $S\setminus \{a\} = \emptyset$, so that the proof of Theorem \ref{theoGL} is complete. 

\section{Minimization of the renormalized energy and the example of tangential boundary data}

\subsection{Renormalized energy of canonical harmonic maps}
Let $\Omega$ be a quasidisk and $\Gamma$ its boundary. In this section, we give an $H^1$ version of the definition of canonical harmonic maps with a prescribed degree $1$ singularity $a \in \Omega$ and of its renormalized energy. The following definition is an $H^1$ version of section I.3 of \cite{BBH}. 
\begin{theodefi}[Canonical harmonic map] \label{canonicalharmonicmap}
Let $a\in \Omega$. Let $g  \in H^{\frac{1}{2}}(\Gamma,\mathbb{S}^1)$ and $deg(g,\Gamma) = 1$. There is a unique map $u : \Omega\setminus \{a\} \to \mathbb{S}^1$ such that
\begin{itemize}
\item $u \in H^1_{loc}(\overline{\Omega}\setminus\{a\})$,
\item $deg(u,a) = 1$,
\item $u = g$ on $\Gamma$,
\item $\star \omega = d\Phi$, where $\omega = \left\langle du, i u \right\rangle$ and $\Phi = \tilde{\Phi} \circ f^{-1}$, where $f:\mathbb{D}\to \Omega$ is a uniformization map such that $f(0)=a$ and $\tilde{\Phi}$ is the unique solution (up to a constant) of
\begin{equation} \label{solutiontildePhi} \begin{cases}
\Delta \tilde{\Phi} = 2\pi \delta_0 & \text{ in } \mathbb{D},\\
\partial_{r} \tilde{\Phi} = \left\langle \partial_{\theta} \left(g\circ f\right), i \left(g\circ f\right) \right\rangle & \text{ on } \mathbb{S}^1.
\end{cases} 
\end{equation}
\end{itemize}
Moreover $u : \Omega\setminus \{a\} \to \mathbb{S}^1$ is a harmonic map. We say that $u$ is the canonical harmonic map of degree $1$ associated to the data $(g,a)$.
\end{theodefi}
\begin{proof}

{\sc Step 1} : Uniqueness of $u$.  Let $u_1$ and $u_2$ satisfying all the properties. Let $U$ be a simply connected domain such that $U \Subset \Omega\setminus \{a\}$. Then, we can write $u_1 = e^{i\psi_1}$ and $u_2 = e^{i\psi_2}$ for $\psi_1,\psi_2 \in H^1(U)$. Then we have that 
$$ d\left( \psi_1 - \psi_2 \right) = \left\langle du_1, i u_1 \right\rangle - \left\langle du_2, i u_2 \right\rangle = \star d\Phi - \star d\Phi = 0 $$
Therefore there is a constant $\alpha_U$ such that $\left\vert \alpha_U \right\vert =1$ and $u_2 =\alpha_U u_1$ in $U$. Since $\Omega$ is connected, we get that $\alpha := \alpha_U$ does not depend on $U$ and $u_2 =\alpha u_1$ in $\Omega$. We finally get that $\alpha = 1$ since on $\Gamma$ we have $u_1 = g = u_2$. 

\medskip

{\sc Step 2} : $\tilde{\Phi}$ is well defined by the Green representation formula. Let $\mathcal{G}_x(y)$ the Green function such that 
\begin{equation} \label{solutiongreenneumann} \begin{cases}
\Delta \mathcal{G}_x = 2\pi \delta_x & \text{ in } \mathbb{D}\\
\partial_{r}  \mathcal{G}_x = 0 & \text{ on } \mathbb{S}^1 \\
\int_{\mathbb{S}^1}  \mathcal{G}_x = 0 \text{ on } \mathbb{S}^1
\end{cases} \end{equation}
Then for $x\in \mathbb{D}$, since $g\circ f \in H^{\frac{1}{2}}\left(\mathbb{S}^1\right)$, $\partial_{\theta} \left(g\circ f\right) \in H^{-\frac{1}{2}}$ and $\mathcal{G}_x \in \mathcal{C}^{\infty}$
$$ \tilde{\Phi}(x) = \mathcal{G}_0(x) + \int_{\mathbb{S}^1} \mathcal{G}_x \left\langle \partial_{\theta} \left(g\circ f\right), i \left(g\circ f\right) \right\rangle $$
is a solution of \eqref{solutiontildePhi}.

\medskip

{\sc Step 3} : Construction of $u$. We pullback all the expected properties of $u$ by $f:\mathbb{D}\to \Omega$, a uniformization map such that $f(0)=a$. We set
$$ \tilde{g} = g\circ f \hspace{5mm} \tilde{u} = u\circ f   \hspace{5mm} \tilde{\omega} =  \left\langle d\tilde{u}, i \tilde{u} \right\rangle $$
Then, $u$ is nothing but $u = \tilde{u}\circ f^{-1}$, where $\tilde{u}$ is the canonical harmonic map of degree $1$ on the disk $\mathbb{D}$ associated to the data $(g\circ f , 0)$ which satisfies the desired property thanks to section 1.3 of \cite{BBH}. 

\end{proof}

\begin{defi}[Renormalized energy of canonical harmonic maps]
Let $u :\Omega \to \mathbb{S}^1$ be the canonical harmonic map with data $a \in \Omega$, $g\in H^{\frac{1}{2}}\left(\Gamma,\mathbb{S}^1\right)$: the renormalized energy is well defined by 
$$ W(a) = \lim_{\delta\to 0}\left( \int_{\Omega \setminus \mathbb{D}_{\delta}(a)} \left\vert \nabla u \right\vert^2 - 2\pi \ln \frac{1}{\delta} \right)$$
\end{defi}

\subsection{Energy minimizing Ginzburg-Landau solutions minimize the renormalized energy} \label{secminimizerenormalized}
Let $u:= u_{\star}$ a minimizing solution of the Ginzburg-Landau problem given by Theorem \ref{theoGL}. We set $v = - i u$, so that $v: \Omega\setminus\{a\} \to \mathbb{S}^1$ is a unit vector field such that, for any point $x\in \Omega\setminus\{a\}$, $(v(x),u(x))$ is an direct orthonormal basis. We set $$ \omega = \left\langle du, v \right\rangle = - \left\langle dv, u \right\rangle$$
the Cartan form in $\Omega$. 

Let $G$ be the Dirichlet Green function with respect to $a$ in $\Omega$, i.e satisfying
\begin{equation} \label{eqgreen} \begin{cases}
\Delta G = 2\pi \delta_a &  \text{in} \ \Omega
\\
G = 0 &  \text{on} \ \Gamma \hskip.1cm. \\
\end{cases} \end{equation}

\begin{prop} \label{cldefoftildemu}
There is a harmonic function $\mu$ in $\Omega$, such that
\begin{equation} \label{eqdefmutilde} \star\omega = d\left(\mu + G\right) \hskip.1cm.  \end{equation}
Then $u$ is the canonical harmonic map of degree $1$ with respect to the data $(g,a)$.
\end{prop}

\begin{proof}
We know that $du \in L^p(\Omega)$ for any $1<p<2$ (see e.g. \cite{Struwe94}, page 1620), so that $\omega\in L^p(\Omega)$. 
We have that
$$  \star d\omega = \star \left(\left\langle du \wedge dv \right\rangle\right) = \left\langle \nabla^{\perp}u,\nabla v \right\rangle = 0 $$
in $\Omega\setminus\{a\}$ since $\left(v,u\right) = i \left(-u,v\right)$. Then the support of the distribution $\star d\omega$ is $\{a\}$.
Since $\omega \in L^p$, we can write
\begin{equation} \label{eqdistorder1} d\omega = \alpha \delta_a  \hskip.1cm. \end{equation}
Let's compute 
$\alpha$.
\begin{equation} \label{eqcomputealpha} \alpha = \int_{\Omega}  d\omega = \int_{\partial\Omega} \omega = \int_{\mathbb{S}^1} \tilde{\omega} = \int_{\partial\Omega} \left\langle \tilde{u}_{\theta},\tilde{v} \right\rangle = 2\pi deg(g) = 2\pi \hskip.1cm. \end{equation}
We then have that
\begin{equation} \label{eqcloseform}
d\left(\omega - \star dG \right) = 0
\end{equation}

Moreover, we prove that there is $\beta\in \mathbb{R}$ such that
\begin{equation} \label{eqstaromegaisclosed} d\left( \star\omega \right) = \beta \delta_a \end{equation} 
in the sense of distributions in $\Omega$. Indeed, again, since $\omega\in L^p(\Omega)$, we just have to prove that the support of the distribution $d\left( \star\omega \right)$ is $\{a\}$.
Indeed, we first have that
$$  \star d \left(\star \omega \right)= \left\langle \Delta u,v \right\rangle + \star \left( \left\langle \star du \wedge dv \right\rangle \right) = \left\langle \Delta u,v \right\rangle + \left\langle  \nabla u,  \nabla v \right\rangle $$
in $\Omega\setminus\{a\}$. Since $u:\Omega\setminus\{a\} \to \mathbb{S}^1$ is a harmonic map we have that 
$$ \left\langle \Delta u,v \right\rangle = 0 $$
in $\Omega\setminus\{a\}$. Moreover, we have that
\begin{equation}\label{eqdvecedvecf} 0 
= \Delta \left( \left\langle u,v \right\rangle \right) 
= \left\langle \Delta u,v \right\rangle + \left\langle u,\Delta v \right\rangle + 2\left\langle \nabla u,\nabla v \right\rangle 
= 2 \left\langle \nabla u,\nabla v \right\rangle \text{ in } \Omega\setminus\{a\} \end{equation}
since $\left\langle u,v \right\rangle = 0$ and $\left(v,u\right) = i \left(-u,v\right)$. We obtain \eqref{eqstaromegaisclosed}.
Then we have that
\begin{equation} \label{eqstaromegaisclosed2} d\left( \star\omega - dG \right) = \beta \delta_a \end{equation} 
which gives that $h = \star\omega - dG - \beta\star dG$ satisfies
\begin{equation} \label{eqstaromegaisclosed3} 
dh = 0 \hbox{ and } d\left(\star h\right) = 0 \hskip.1cm.
 \end{equation} 
Then $h$ is a smooth harmonic form in $\Omega$. There is a smooth function $\tilde{\mu}$ such that $h= d\tilde{\mu}$. To complete the proof of Claim \ref{cldefoftildemu}, we finally prove that $\beta=0$.
From \ref{eqenergyestimatefara} we know that there is a constant $C>0$ such that for $\delta < \frac{dist(a,\Gamma)}{2}$,
$$ \int_{\Omega \setminus B_{a}(\delta)} \left\vert \nabla u \right\vert^2 \leq 2\pi \ln\left(\frac{1}{\delta}\right) + C \hskip.1cm.$$
Knowing that $\left\vert \nabla u\right\vert^2 = \left\vert \left\langle \nabla u, v\right\rangle \right\vert^2 = \left\vert \omega \right\vert^2$, we have
\begin{equation} \label{ineqenergyminimizingGL} 2\pi \ln\left(\frac{1}{\delta}\right) + C \geq \int_{\Omega \setminus B_{a}(\delta)}\left\vert \omega \right\vert^2 = \int_{\Omega \setminus B_{a}(\delta)}\left\vert \nabla G + \beta\nabla^{\perp}G + \nabla \mu \right\vert^2 \hskip.1cm. \end{equation}
Since $\mu$ is smooth on $\Omega$ and $\nabla G$ is bounded in $L^{p}$ for any $p<2$, we have a constant $C'$ independent from $\delta < \frac{d(a,\Gamma)}{2}$ such that
$$\left\vert \int_{\Omega \setminus B_{a}(\delta)} 2\left\langle  \nabla G + \beta\nabla^{\perp}G ,\nabla\mu \right\rangle +\left\vert \nabla \mu \right\vert^2 \right\vert \leq C' $$
Then
$$ \int_{\Omega \setminus B_{a}(\delta)}\left\vert \nabla G + \beta\nabla^{\perp}G  \right\vert^2 =  \left(1+ \beta^2\right) \int_{\Omega \setminus B_{a}(\delta)}\left\vert \nabla G \right\vert^2 + 2\beta \int_{\Omega \setminus B_{a}(\delta)}\left\langle \nabla G,\nabla^{\perp} G \right\rangle$$
Since $\left\langle \nabla G,\nabla^{\perp} G \right\rangle = 0$, we then can write \eqref{ineqenergyminimizingGL} as
\begin{equation} \label{ineqenergyminimizingGL2} 2\pi \ln\left(\frac{1}{\delta}\right) + C \geq \left(1+\beta^2\right) \int_{\Omega \setminus B_{a}(\delta)}\left\vert \nabla G \right\vert^2 - C' \geq 2\pi \left(1+\beta^2\right)\ln\left(\frac{1}{\delta}\right) - C'' \hskip.1cm, \end{equation}
for a constant $C''$ independent from $\delta$, since $G(x) = \ln\left\vert x-a \right\vert + w(x) $ on $\Omega$, where $w$ is a harmonic function. Dividing by $\ln\left(\frac{1}{\delta}\right)$ and letting $\delta \to 0$ gives $\beta = 0$. 

Now it is easy to check that $u$ is the canonical harmonic map of degree $1$ with respect to the data $(g,a)$ and the proof of Claim \ref{cldefoftildemu} is complete.
\end{proof}

Finally, we also have, in our setting, the fact that the singularity of the minimizer minimizes the renormalized energy, the proof is exactly the same as the one of Theorem VIII.1 \cite{BBH}.

\begin{prop} \label{clminimizerrenormalizedenergy} 
The singularity $a$ of the minimizing Ginzburg-Landau solution $u$ is a minimizer of the renormalized energy $W(a)$.
\end{prop}

\subsection{The special case of tangential boundary data} \label{linkrenormalizeduniformization}
In this section we prove the Proposition \ref{propWeylPeterssonframeenergy}, when the boundary data $g=\tau$, where $\tau$ is the tangent vector field on $\Gamma$ such that the frame $(\nu, \tau)$ is a direct frame, where $\nu$ is the out-pointing normal of $\Omega$. We have $\tau = -i \nu$.

In the following proposition \ref{propcanonicaluniformization} we aim at proving that canonical harmonic maps with respect to tangential and normal data are exactly push forwards in $\Omega$ of the polar orthonormal frame of the disk (with singularity at $0$) by uniformization maps.

Let $f : \mathbb{D} \to \Omega$ be a uniformization map of $\Omega$. We set $u : \Omega \to \mathbb{S}^1$
$$ u = \frac{f_{\theta}}{\left\vert f_{\theta} \right\vert} \circ f^{-1} $$
and $ v: \Omega \to \mathbb{S}^1$
$$ v = \frac{f_{r}}{\left\vert f_{r} \right\vert} \circ f^{-1} $$
where $f_r$ and $f_{\theta}$ are the radial and angular derivatives of $f$ defined by
$$f_r := \partial_r f =  \frac{x f_x + y f_y}{\sqrt{x^2+y^2}} = \frac{x f_x + y f_y}{r}$$
where $r = \sqrt{x^2+y^2}$ is the radial coordinate and
$$f_{\theta} := \partial_{\theta} f = -y f_x + x f_y$$
where $f_x := \partial_x f$ and $f_y := \partial_y f$ are the partial derivatives with respect to the coordinates of $z = x+i y$.

\begin{prop} \label{propcanonicaluniformization}
Let $a\in \Omega$. For any uniformization map $f : \mathbb{D} \to \Omega$ such that $f(0) = a$, $u = \frac{f_{\theta}}{\left\vert f_{\theta} \right\vert} \circ f^{-1}$ and $v = \frac{f_{r}}{\left\vert f_{r} \right\vert} \circ f^{-1}$ are the canonical harmonic maps of degree $1$ associated to the data $(\tau, a)$ and $(\nu, a)$. 
\end{prop}

\begin{proof}

Let $f : \mathbb{D} \to \Omega$ be a uniformization map such that $f(0) = a$. Notice that since $f$ is a conformal map, $(v,u)$ is an orthonormal frame in $\Omega$ and notice also that on $\Gamma$, $u = \tau$ and $v = \nu$. We now set 
$$ \tilde{\omega} = \left\langle d\tilde{u},\tilde{v} \right\rangle  = - \left\langle d\tilde{v},\tilde{u} \right\rangle $$
the Cartan form of $(\tilde{v},\tilde{u})$ where
$$  \tilde{v} = \frac{f_{r}}{\left\vert f_{r} \right\vert} \text{ and } \tilde{u} = \frac{f_{\theta}}{\left\vert f_{\theta} \right\vert} $$
and
$$ \tilde{\mu} = \ln  \left\vert f_x \right\vert = \ln \left\vert f_y \right\vert = \ln \left\vert f_r \right\vert = \ln \left( \frac{\left\vert f_\theta \right\vert}{r} \right)$$
so that
$$ \begin{cases}
f_r = e^{\tilde{\mu}} \tilde{v} \\
f_{\theta} = e^{\tilde{\mu} + \tilde{G}} \tilde{u}
\end{cases} $$
and
$$ \begin{cases}
\tilde{v} = \left(x f_x + y f_y\right) e^{-\tilde{\mu}-\tilde{G}}  \\
\tilde{u} = \left(-y f_x + x f_y\right) e^{-\tilde{\mu} -\tilde{G}} 
\end{cases} $$
where $\tilde{G}(x,y) = \ln r$ is the Green function of $\mathbb{D}$ with Dirichlet boundary condition with respect to the singularity $0$. 

Let's prove the following equation:
\begin{equation} \label{eqstaromega} \star \tilde{\omega} = d\left( \tilde{\mu} + \tilde{G} \right) \end{equation}
in $\mathbb{D}$. We have that
\begin{equation*} \begin{split}\left\langle\tilde{u}_x , \tilde{v} \right\rangle 
& = \left\langle -y f_{xx} + f_y + x f_{xy} , x f_x + y f_y \right\rangle e^{-2\tilde{\mu} - 2\tilde{G}} \\
& = \left(-xy f_{xx}.f_x - y^2 f_{xx}.f_y + x f_y.f_x + y \left\vert f_y \right\vert^2 + x^2 f_{xy}.f_x + xy f_{xy}.f_y\right) e^{-2\tilde{\mu} - 2\tilde{G}} \\
& = \left(-\frac{xy}{2} \left( \left\vert f_x \right\vert^2 \right)_x + \frac{y^2}{2} \left( \left\vert f_y \right\vert^2 \right)_y + y \left\vert f_y \right\vert^2 + \frac{x^2}{2} \left( \left\vert f_x \right\vert^2 \right)_y + \frac{xy}{2} \left( \left\vert f_y \right\vert^2 \right)_x\right)e^{-2\mu - 2\tilde{G}} \\
& = \left( -xy \tilde{\mu}_x + y^2 \tilde{\mu}_y + y + x^2 \tilde{\mu}_y + xy \tilde{\mu}_x \right) e^{-2\tilde{G}} \\
& = \tilde{\mu}_y + \tilde{G}_y
  \end{split}
  \end{equation*}
where we used that $f_{xx} = -f_{yy}$ since $f$ is harmonic $\Delta f = 0$.
We also have that
\begin{equation*} \begin{split}\left\langle\tilde{u}_y , \tilde{v} \right\rangle 
& = \left\langle -f_x - y f_{xy} + x f_{yy} , x f_x + y f_y \right\rangle e^{-2\tilde{\mu} - 2\tilde{G}} \\
& = \left(- x \left\vert f_x \right\vert^2 - y f_x.f_y  - xy f_{xy}.f_x - y^2 f_{xy}.f_y  + x^2 f_{yy}.f_x + xy f_{yy}.f_y\right) e^{-2\tilde{\mu} - 2\tilde{G}} \\
& = \left(- x \left\vert f_x \right\vert^2  -\frac{xy}{2} \left( \left\vert f_x \right\vert^2 \right)_y - \frac{y^2}{2} \left( \left\vert f_y \right\vert^2 \right)_x  - \frac{x^2}{2} \left( \left\vert f_x \right\vert^2 \right)_x + \frac{xy}{2} \left( \left\vert f_y \right\vert^2 \right)_y \right) e^{-2\tilde{\mu} - 2\tilde{G}} \\
& = \left(- x - xy \tilde{\mu}_y - y^2 \tilde{\mu}_x - x^2 \tilde{\mu}_x + xy \tilde{\mu}_y \right) e^{-2\tilde{G}} \\
& = - \left(\tilde{\mu}_x + \tilde{G}_x\right)
  \end{split}
  \end{equation*}
We proved \eqref{eqstaromega}. Moreover, since $(\tilde{u},\tilde{v})$ is an orthonormal frame, we have
$$  \star d\tilde{\omega} = \star \left(\left\langle d\tilde{u} \wedge d\tilde{v} \right\rangle\right) = \left\langle \nabla^{\perp}\tilde{u},\nabla \tilde{v} \right\rangle = 0 \text{ in }\mathbb{D}\setminus\{0\}$$
 since $\left(\tilde{u},\tilde{v}\right) = i \left(\tilde{v},-\tilde{u}\right)$. Then the support of the distribution $\star d\tilde{\omega}$ is $\{0\}$.
Since by \eqref{eqstaromega} $\tilde{\omega} \in L^p_{loc}\left(\mathbb{D}\right)$ for $1\leq p < 2$, we can write
\begin{equation} \label{eqdistorder0} d\tilde{\omega} = \alpha \delta_0  \hskip.1cm, \end{equation}
where $\delta_0$ is the Dirac mass supported in $\{0\}$. Let's compute 
$\alpha$.
\begin{equation} \label{eqcomputealphadegree} \alpha = \int_{\mathbb{D}}  d\tilde{\omega} = \int_{\mathbb{S}^1} \tilde{\omega} = \int_{\mathbb{S}^1} \left\langle \tilde{u}_{\theta},\tilde{v} \right\rangle = 2\pi deg(\tilde{u},\mathbb{S}^1) = 2\pi \hskip.1cm. \end{equation}
This degree is well defined since $\tilde{u} \in H^{\frac{1}{2}}(\mathbb{S}^1)$ because $\Gamma$ is Weil-Petersson. We then have that
\begin{equation} \label{eqomegagreen}
\Delta \tilde{\mu} = d\left(\tilde{\omega} - \star d\tilde{G} \right) = 0 \hskip.1cm.
\end{equation}

\medskip

In $\Omega$, we have $u = \tilde{u}\circ f^{-1}$ and $v = \tilde{v}\circ f^{-1}$. We set 
$$\omega = \left\langle d u , v \right\rangle = \left(f^{-1}\right)' \tilde{\omega} \circ f^{-1} $$
the Cartan form on $\Omega$ and $G = \tilde{G} \circ f^{-1}$ the Dirichlet-Green function with respect to $a:= f(0)$. We obtain from \eqref{eqstaromega}
\begin{equation} \label{eqstaromegatilde} \star \omega = d\left( \mu + G \right) \end{equation}
denoting $ \mu = \tilde{\mu} \circ f^{-1}$, where we have that $ \tilde{\Phi} = \left(\mu + G\right)\circ f $ is a solution of \eqref{solutiontildePhi}. 

\medskip

Notice that $\Gamma$ is a Weil-Petersson curve if and only if
$$\left\vert \nabla \tilde{\mu} \right\vert^2 = \left\vert \frac{f''}{f'}   \right\vert^2 \in L^1(\mathbb{D}) $$
if and only if $\left\vert \nabla \mu \right\vert^2 \in L^1\left(\Omega\right)$ where $\mu = \tilde{\mu} \circ f^{-1}$ by conformal invariance of the Dirichlet energy.
 In particular by \eqref{eqstaromegatilde}, $u \in H^1_{loc}\left( \overline{\Omega}\setminus \{a\} \right)$ and $\omega \in L^2_{loc}\left( \overline{\Omega}\setminus \{a\} \right)$. 
 Therefore, we have all the necessary regularity to have $u = \tau \in H^{\frac{1}{2}}(\Gamma)$ and the proof of the proposition is complete.
\end{proof}

\begin{rem} Notice that by Proposition \ref{cldefoftildemu}, $\tilde{\mu}= \mu \circ f$ satisfies the following equation
\begin{equation} \label{eqonmu}
\begin{cases}
\Delta \tilde{\mu} = 0 & \text{ in } \ \mathbb{D} \\
\partial_{r} \tilde{\mu} = \left\langle \partial_{\theta}\tilde{u},\tilde{v} \right\rangle  - 1 & \text{ on } \ \partial{\mathbb{D}} \hskip.1cm,
\end{cases}
\end{equation}
where $\partial_r$ and $\partial_\theta$ denote the radial and angular derivatives on the disk, and $\tilde{u} = u\circ f$, $\tilde{v}= v\circ f$. Equation \eqref{eqonmu} is nothing but the Liouville equation up to the boundary, since $\left\langle u_{\tau},v \right\rangle$ is the curvature of $\Gamma$,see \cite{CL} for more details about prescribing geodesic curvature.
\end{rem}

Now, we can reformulate the renormalized energy in terms of the uniformization

\begin{prop} \label{renormenerunif} Let $f : \mathbb{D} \to \Omega$ be a uniformization map such that $f(0) = a$. Then
\begin{equation} \label{renormalizedenergy} W(a) = W_0 - 2\pi \ln \left\vert f'(0) \right\vert \end{equation}
where 
$$ W_0 = \int_{\mathbb{D}} \left\vert \frac{f''}{f'} \right\vert^2 + 4\pi \ln \left\vert f'(0) \right\vert $$
does not depend on the choice of the uniformization $f : \mathbb{D}\to \Omega$ nor in $a$. 
\end{prop}

\begin{proof}
The fact that $W_0$ only depends on $\Omega$ is well known, see lemma 3.4 page 86 of \cite{tate}.
As in the proof of Proposition \ref{propcanonicaluniformization}, we let $\tilde{\mu} = \ln \left\vert f' \right\vert$. Since $\Gamma$ is a Weil-Petersson curve,
$$\left\vert \nabla \tilde{\mu} \right\vert^2 = \left\vert \frac{f''}{f'}   \right\vert^2 \in L^1(\mathbb{D}) $$
so that $\left\vert \nabla \mu \right\vert^2 \in L^1\left(\Omega\right)$ where $\mu = \tilde{\mu} \circ f^{-1}$ by conformal invariance of the Dirichlet energy. In particular by \eqref{eqstaromegatilde}, $\omega \in L^2_{loc}\left( \overline{\Omega}\setminus \{a\} \right)$. 
Then by a simple computation and by \eqref{eqstaromegatilde},
\begin{equation*}
\begin{split}
\int_{\Omega\setminus B_{\rho}(a)} \left\vert \omega \right\vert^2 =
& \int_{\Omega\setminus B_{\rho}(a)} \left\vert \nabla G \right\vert^2 + 2 \int_{\Omega\setminus B_{\rho}(a)} \left\langle \nabla G,\nabla \mu \right\rangle + \int_{\Omega\setminus B_{\rho}(a)} \left\vert \nabla \mu \right\vert^2 \\
= & - \int_{\partial B_{\rho}(a)} G \partial_{\nu}G  - 2 \int_{\partial B_{\rho}(a)} G\partial_{\nu} \mu  + \int_{\Omega\setminus B_{\rho}(a)} \left\vert \nabla \mu \right\vert^2 \\
= & - \int_{\partial B_{\rho}(a)} \left(\ln \left\vert x-a \right\vert + h\right) \partial_{\nu}\left( \ln \left\vert x-a \right\vert\right) + O\left(\rho \ln \rho \right) \\
& + \int_{\Omega\setminus B_{\rho}(a)} \left\vert \nabla \mu \right\vert^2 \\
= & 2\pi \ln \left(\frac{1}{\rho}\right) + \int_{\Omega\setminus B_{\rho}(a)} \left\vert \nabla \mu \right\vert^2 - 2\pi  h(a) + o(1)
\end{split}
\end{equation*}
as $\rho\to 0$, where $h(x) = G(x,a) -  \ln \left\vert x-a \right\vert $ is a smooth harmonic function in $\Omega$. Since $\ln \left\vert x \right\vert = G(f(x),a)$, we have that
\begin{equation} \label{eqdefmass}
\begin{split} h(a) & = h( f(0) ) \\
& = \lim_{z\to 0} \left( G(f(z)) -  \ln \left\vert f(z)-a \right\vert \right) \\
& = \lim_{z\to 0}\left( -  \ln\left\vert \frac{f(z)-a}{z} \right\vert \right) \\
& = - \ln \left\vert f'(0) \right\vert
\end{split} 
\end{equation}
so that
$$ W(a) = \int_{\mathbb{D}} \left\vert \nabla \tilde{\mu} \right\vert^2 + 2\pi \ln \left\vert f'(0) \right\vert \hskip.1cm. $$
We then have the expected identity \eqref{renormalizedenergy}.
\end{proof}

\begin{rem}

A minimizer $a \in \Omega$ of $W$ is $a = f(0)$ where $f$ is a maximizer of the following maximization problem
$$ \max_{f : \mathbb{D}\to \Omega} \left\vert f'(0) \right\vert  $$
where the maximum stands on the set of uniformizing maps of $\Omega$. Or equivalently, fixing one uniformization map $f_0 : \mathbb{D}\to \Omega$, a minimizer is $a = f_0(\omega_0)$ where $\omega_0 \in \mathbb{D}$ is chosen as a maximizer for
\begin{equation} \label{eqmaximizationproblem} \left\vert f_0'(\omega_0) \right\vert \left(1-\left\vert \omega_0 \right\vert^2\right) = \max_{\omega \in \mathbb{D}} \left\vert f_0'(\omega) \right\vert \left(1-\left\vert \omega \right\vert^2\right) \hskip.1cm. \end{equation}
Indeed, we can write any uniformization map as $f = f_0 \circ \psi^{-1}$, where
$$ \psi(z) = e^{i\theta} \frac{z-\omega}{\bar{\omega}z - 1} $$
for $\omega \in \mathbb{D}$ and $\theta\in \mathbb{R}$ so that since $\psi(\omega) = 0$
$$ \left\vert f'(0) \right\vert = \left\vert f_0'(\omega)  \left(\psi^{-1}\right)'(0) \right\vert = \left\vert \frac{f_0'(\omega)}{ \psi'(\omega)} \right\vert = \left\vert f_0'(\omega) \right\vert  \left(1-\left\vert \omega \right\vert^2\right) \hskip.1cm. $$

As we can see in \eqref{eqdefmass}, $- \ln \left\vert f'(0) \right\vert$ is nothing but the mass of the Dirichlet-Green function of $\Omega$ at the point $a$ where $f$ is defined as a uniformization map such that $f(0) = a$. 

\end{rem}

\begin{rem}
We know by section \ref{secinteriorsingularity} that a minimizer of the renormalized energy cannot be realized on the boundary of $\Omega$. 
In the case of tangential data, it is a simple consequence of the previous remark and Koebe's distortion theorem. 
Indeed, a point $\omega_0$ satisfying \eqref{eqmaximizationproblem} must belong to $\mathbb{D}$ since 
$$\left\vert f_0'(\omega) \right\vert \left(1-\left\vert \omega \right\vert^2\right)  \leq dist(f_0(\omega),\partial\Omega) \to 0 \text{ as } \left\vert \omega \right\vert \to 1$$
so that $a = f_0(\omega_0)\in \Omega$.
\end{rem}

\section{A new point of view on the Riemann mapping theorem} \label{Riemannmappingsmooth}

We assume in this section that $\Gamma$ is a smooth\footnote{ In fact we can assume $C^2$ but not much less as explained in remark \ref{frem} below.} Jordan curve. As a side remark, we aim at proving the Riemann mapping theorem for the smooth bounded domain $\Omega$ enclosing $\Gamma$. 

Let $(v,u)$ be the frame defined in subsection \ref{secminimizerenormalized} with tangential and normal boundary data, and $\omega = \left\langle du,v \right\rangle$ be the associated Cartan form. We let $\Phi = \mu+G$ where $G$ is defined by \eqref{eqgreen} and $\mu$ is given by Proposition \eqref{cldefoftildemu}. Notice that since we assumed that $\Omega$ is smooth, we never needed any uniformization map of $\Omega$ to build the moving frame $(v,u)$, since the map $u$ is given from the classical Ginzburg-Landau setting. We compute the following Lie bracket in $ \Omega\setminus \{a\}$:
\begin{eqnarray*}
\left[e^{\Phi} v ,e^{\Phi} u  \right] &=& e^{2\Phi}\left\{ \left[ v, u  \right] + d\Phi\left(v\right)u - d\Phi\left(u\right)v \right\} \\
&=& e^{2\Phi}\left\{ \nabla_{v} u - \nabla_{u} v + d\Phi\left(v\right)u - d\Phi\left(u\right)v \right\} \\
& = & e^{2\Phi}\left\{ \left(   \omega\left(u\right) + d\Phi\left(v\right)   \right) u + \left( - d\Phi\left(u\right) + \omega\left(v\right) \right) v \right\} \hskip.1cm. \\
\end{eqnarray*}
Let $(\alpha, \beta)$ the dual basis of $(v,u)$ in $T^{\star} \Omega$, we have that
$$ \star \omega =\star\left( \omega\left(v\right)\alpha + \omega\left(u\right)\beta \right) = \omega\left(v\right)\beta - \omega\left(u\right)\alpha \hskip.1cm,$$
where $\star \alpha = \beta $ and $\star \beta = -\alpha$. Therefore, we have that
$$ \left[e^{\Phi} v,e^{\Phi} u  \right] =  e^{2\Phi}\left\{  \left( - \star\omega\left(v\right)  + d\Phi\left(v\right)  \right) u + \left( - d\Phi\left(u\right) + \star\omega\left(u\right) \right) v \right\} = 0 \hskip.1cm,$$
since $\star \omega = d\Phi$. 
We have the following result, which is a global version of the Frobenius theorem in our context

\begin{prop} \label{Frobenius} There is a homeomorphism $\psi: ]-\infty,0]\times \mathbb{R}/\rho\mathbb{Z} \to \overline{\Omega}\setminus \{a\}$ such that
\begin{equation} \label{eqfrobeniusglobal} \begin{cases}
\partial_s \psi = e^{\Phi\circ\psi} v\circ\psi \\
\partial_{\theta} \psi = e^{\Phi\circ\psi} u \circ\psi \hskip.1cm. \\
\end{cases} \end{equation}
In particular, $\psi: ]-\infty,0[\times \mathbb{R}/\rho\mathbb{Z} \to \Omega\setminus \{a\}$ is a conformal diffeomorphism.
\end{prop}

\begin{proof}

We know by the Frobenius theorem (see for instance \cite{Warner1983}), that for any $x \in \Omega\setminus \{a\}$, there is a map $\psi_U : U \to \Omega \setminus\{a\}$, defined on an open set $U$ such that $0\in U$, $\psi_U(0)=x$ and $\psi(s,\theta)$ satisfies
\begin{equation} \label{eqfrobeniuslocal} \begin{cases}
\partial_s \psi_U = e^{\Phi\circ\psi_U} v\circ\psi_U \\
\partial_{\theta} \psi_U = e^{\Phi\circ\psi_U} u \circ\psi_U \hskip.1cm. \\
\end{cases} \end{equation}
Moreover, $\psi_U$ is unique in the sense that two such maps are equal on the intersection of there open sets $U$ of definition. Notice also that such an application $\psi_U$ is a local diffeomorphism. Let $(U_x,\psi_x)$ be the maximal solution: $U_x$ is defined as the union of open sets $U$ such that there is a map $\psi_U$ satisfying $\psi_U(0)=x$ and \eqref{eqfrobeniuslocal}. By uniqueness in the Frobenius theorem, $\psi_x(y) = \psi_U(y)$ is perfectly defined independently of $U$ and satisfies $\psi_x(0)=x$ and \eqref{eqfrobeniuslocal}. We aim at proving that up to factorisation and translation on the domain of definition, $\psi_x$ is the map we are looking for.

\medskip

All along the proof we denote $\varphi_{X_1}$ and $\varphi_{X_2}$ the flows along the vector fields $X_1 := e^{\Phi} v$ and $X_2 := e^{\Phi} u$. Notice that the partial maps of $\psi_x$ are paths along the flow: if $x_0 = \psi_{x}(s_0,\theta_0)$,
$$\psi_{x}(s_0 + s,\theta_0) = \varphi^s_{X_1}(x_0) \hbox{ and }  \psi_x(s_0,\theta_0 + \theta) = \varphi^{\theta}_{X_2}(x_0)\hskip.1cm.$$
Notice also that we can extend $X_1$ and $X_2$ by smooth vector-fileds on an open set $\widetilde{\Omega}$ such that $\overline{\Omega} \subset \widetilde{\Omega}$, and that there is no stationary point on $\widetilde{\Omega}\setminus \{a\}$ for $\varphi_{X_1}$ and $\varphi_{X_2}$. In particular, the flows are perfectly defined up to the boundary in $\overline{\Omega}$. We recall that $\partial \Omega$ is a cycle of $\varphi_{X_2}$ by assumption. Moreover, the only stationary point of $\varphi_{X_1}$ and $\varphi_{X_2}$ is $a$ since by definition of $\Phi$ as a Green function on $a$, the only vanishing point of $e^{\Phi}$ is $a$.

\medskip

{\sc Step 1} : \textit{We prove that for any $x$, $\psi_x$ is surjective}

\medskip

Notice that if $x,y \in \Omega\setminus\{a\}$, we have that
$$ \psi_x(U_x) \cap \psi_y(U_y) \neq \emptyset \Rightarrow \psi_x(U_x) = \psi_y(U_y) \hskip.1cm.$$
Indeed, assume that $z = \psi_x(u) = \psi_y(v)$ for $u\in U_x$ and $v\in U_y$. By uniqueness in the Frobenius theorem, and by maximality of $(U_x,\psi_x)$ and $(U_y, \psi_y)$, one has $U_x = U_y + \{b-a\}$ and $\psi_x(z) = \psi_y(z+v-u)$ for any $z\in U_x$. Since for any $y\in \Omega$, $\psi_y$ is a local diffeomorphism, then if $y\in \psi_x(U_x)$, there is an open set $U$ such that $\psi_y(U)$ is an open neighbourhood of $y$ in $\psi_x(U_x)$. Therefore $\psi_x(U_x)$ is open. Denoting $C= \{C(x); x\in \Omega\setminus\{a\}\}$, where 
$$C(x) = \{y\in \Omega\setminus\{a\} ; \psi_x(U_x)= \psi_y(U_y)\} \hskip.1cm,$$ 
we can write
$$ \Omega\setminus\{a\} = \bigcup_{C(x) \in C} \psi_x(U_x) $$
as a disjoint union of open sets. By connectedness, the cardinal of $C$ is $1$ and $\psi_x(U_x) = \Omega\setminus\{a\}$ for any $x\in \Omega\setminus\{a\}$.

\medskip

From now on, we drop the index $x$ in the definition of $U := U_x$ and $\psi:= \psi_x$

\medskip

{\sc Step 2} : \textit{We prove the existence of $\beta$ such that $U = (-\infty,\beta) \times \mathbb{R}$ and such that we can extend $\psi$ on $(-\infty,\beta] \times \mathbb{R}$ so that $\psi \left( \{\beta\} \times \mathbb{R} \right) = \partial\Omega$ and $\psi$ is a local homeomorphism up to the boundary.}

\medskip

We first prove that for any $(s_0,\theta_0)\in U$, we have $\{s_0\} \times \mathbb{R} \subset U$. Indeed, let $(s_0,\theta_0) \in U$, $x_0 = \psi(s_0,\theta_0)$, and let $(\gamma,\delta)$ be the maximal interval such that $\gamma <\theta_0 < \delta$ and $\{s_0\} \times (\gamma,\delta) \subset U$. Let's prove that $\delta= +\infty$ (proving that $\gamma=-\infty$ is analogous). 

Since $\overline{\Omega}$ is compact, up to a subsequence of $\theta_n \to \delta$ one has $\lim_{\theta_n \to \delta} \psi(s,\theta_n) = l \in \overline{\Omega}$. If $l\in \Omega\setminus\{a\}$ and $\delta<+\infty$, by the Frobenius theorem, one can extend $\psi$ by a map $\hat{\psi}$ satisfying $\hat{\psi}(s,\delta) = l$ and \eqref{eqfrobeniuslocal}, defined on an open neighbourhood of $(s,\delta)$ and we contradict the maximality of $\delta$. Therefore, $l\in \partial\Omega \cup \{a\}$ or $\delta = +\infty$. Now, we recall that $\psi(s_0,\theta) = \varphi^{\theta-\theta_0}_{X_2}(x_0)$ is the flow associated to $X_2$. If $l=a$, then $\lim_{\theta \to \delta} \psi(s,\theta) = a$ and since $a$ is a stationary point, we have that $\delta = +\infty$. Similarly, if $l\in \partial\Omega$, since $\partial\Omega$ is a cycle of $\varphi_{X_2}$, we must have $\delta = +\infty$. In all the cases, $\delta = +\infty$.

Now, let $\beta = \sup\{ s\in \mathbb{R} ; \{s\} \times \mathbb{R} \subset \Omega \}$. 

We prove that $\beta<+\infty$ and that for any $\theta\in \mathbb{R}$, $\lim_{z \to (\beta,\theta)} \psi(y)$ exists and belongs to $\partial\Omega$. 

The vector field $X_1 = e^{\Phi} v$ is transversal to $\partial\Omega$ (and outward pointing on $\partial{\Omega}$). Let $y\in \partial \Omega$. Let $F : ]-\eps,\eps[ \times ]-\eps,\eps[ \to V$ be a straightening of $X_1$ at the neighbourhood $V$ of $y = F(0,0)$: ie $F$ is a diffeomorphism and $F^{\star}(X_1) = \partial_{x_1}$. We have that $F(]-\eps,0[\times ]-\eps,\eps[ ) \subset \Omega$. Let $(s_0,\theta_0) \in U$ be such that $x_0 = \psi(s_0,\theta_0) = F(-\frac{\eps}{2},0)$. Recalling that $\psi(s,\theta_0) = \varphi_{X_1}^{s-s_0}(x_0)$, we deduce the expected result.

Now, we let $\alpha = \inf \{ \tilde{\alpha} \in (-\infty,\beta) ; (\tilde{\alpha},\beta)\times \mathbb{R}  \subset U\}$. With the same methods as previously thanks to the minimality of $\alpha$ and using the flow associated to $X_1$ by $\psi(s,\theta) = \varphi^{s-s_0}_{X_1}(x_0)$ one can prove that $\alpha = -\infty$. This ends the proof of STEP 2.

\medskip

{\sc Step 3} : \textit{We prove the existence of $\rho>0$ such that $\psi(s,\theta) = \psi(s,\theta+\rho)$ for any $(s,\theta)\in U$.}

\medskip

It is equivalent to prove that for any $x\in \psi(U)$, we have that $\varphi_{X_2}(x)$ is a cycle.

Let $y\in \partial \Omega$. Since $\varphi_{X_2}(y)$ is a cycle, there is a minimal $\rho>0$ such that $\varphi_{X_2}^{\rho}(y) = y $. We aim at proving that $\varphi_{X_2}^{\rho}(x) = x $ for any $x\in \psi(U) $.

Let $x_0\in \psi(U)$ and $(s_0,\theta_0) \in U $ be such that $x_0= \psi(s_0,\theta_0)$.  From the computations we made before Proposition \ref{Frobenius}, we have that $[X_1,X_2] = 0$ in $\Omega\setminus \{a\}$. Therefore, the flows $\varphi_{X_1}$ and $\varphi_{X_2}$ commute. Then by STEP 2,
$$ \varphi_{X_2}^{\rho}(x_0) = \varphi_{X_1}^{s_0-s} \circ \varphi_{X_2}^{\rho} \circ \varphi_{X_1}^{s-s_0}(x_0) \hskip.1cm,$$
for any $s_0 \leq s<\beta$. By STEP 2 again, we can pass to the limit as $s\to \beta$ and we get
\begin{equation} \label{eqcommutationflowboundary} \varphi_{X_2}^{\rho}(x_0) = \varphi_{X_1}^{s_0-\beta} \circ \varphi_{X_2}^{\rho} \circ \varphi_{X_1}^{\beta-s_0}(x_0) \hskip.1cm. \end{equation}
Since $\varphi_{X_1}^{\beta-s_0}(x_0) \in \partial\Omega$, we have that $\varphi_{X_2}^{\rho} \left( \varphi_{X_1}^{\beta-s_0}(x_0) \right) = \varphi_{X_1}^{\beta-s_0}(x_0)$. We then get STEP 3 from \eqref{eqcommutationflowboundary}.

\medskip

{\sc Step 4} : \textit{We prove the Proposition}

\medskip

We assume that $\rho>0$ given by STEP 3 is minimal. Now we can factorize $\psi$ into a surjective map we still denote $\psi : (-\infty,\beta]\times \mathbb{R}/\rho\mathbb{Z} \to \overline{\Omega}\setminus\{a\}$.

Let's prove that this map is now injective. Let $z_1 = (s_1,\theta_1)$ and $z_2 = (s_2,\theta_2)$ such that $\psi(z_2) = \psi(z_1)$. Then $\psi(z) = \psi(z+z_2-z_1)$. Assume that $s_1\leq s_2$. We then have that $\psi(\{\beta - s_2 +s_1\} \times\mathbb{R}) = \partial\Omega $. Then we must have $s_1 = s_2$. By minimality of $\rho$, we get that $\theta_1 = \theta_2$. Then $\psi$ is injective. Since $\psi$ is bijective and a local homeomorphism, it is a global homeomorphism. 

Moreover, since $\psi : (-\infty,\beta)\times \mathbb{R}/\rho\mathbb{Z} \to \Omega\setminus\{a\}$ is bijective smooth and open, it is a diffeomorphism. By equation \eqref{eqfrobeniuslocal}, it is also a conformal diffeomorphism. Finally, we replace $\psi$ by $\psi(z-(0,\beta))$ and we complete the proof of Proposition \ref{Frobenius}.

\end{proof}

We set a holomorphic function $f:\mathbb{D}\setminus\{0\} \to \Omega \setminus \{a\}$, defined by
\begin{equation} \label{deff} f(re^{i\theta}) = \psi\left( \frac{\rho}{2\pi} \left(\ln r,\theta \right) \right) \hskip.1cm, \end{equation}
where $\psi$ is the map given by claim \eqref{Frobenius}. The implicit formula $f(e^{\frac{2\pi w}{\rho}}) = \psi(w)$ holds and one can directly extend $f$ to a bi-holomorphic map $f:\mathbb{D}\to \Omega$ by $f(0) = a$. This ends the proof of the Riemann mapping theorem.

\begin{rem}
\label{frem}
We cannot ask for low regularity of $\Omega$ to prove the result by this method for two reasons : $\mathcal{C}^1$ regularity of $\Omega$ is needed to use classical construction of Ginzburg-Landau minimizers (without using the uniformization map as in the current paper for Weil-Petersson curves). Moreover, we use that the unit tangent vector field on $\Gamma$ is Lipschitz in order to apply the Cauchy-Lipschitz theorem and define the flow of vector fields. Using some methods from dynamical systems, it may be possible to prove the existence of a closed orbit in the interior, which would play the same role as the boundary orbit and would permit to conclude. 
\end{rem}

\appendix
\section{Technical lemmas}
The result of this section are well known, we prove it in a way more adapted to our paper for the sake of completeness. First we need a basic estimate on the Poisson Kernel
\begin{lem}
Let $M>0$, then 

\begin{equation}
\label{eqM}
\lim_{r\rightarrow 1^{-}}(1-r) \int_\delta^{2\pi- M(1-r)} \frac{1}{1+r^2-2r\cos(\theta)} \,  d\theta = 8\arctan \left(\frac{1}{M}\right).
\end{equation}
\end{lem}

\begin{proof} We fix $M>0$ and we set $\delta=M(1-r)$, then we have 
\begin{align*}
(1-r) \int_\delta^{2\pi-\delta} \frac{1}{1+r^2-2r\cos(\theta)} \,  d\theta =&(1-r)\int_{-\pi+\delta}^{\pi-\delta} \frac{1}{1+r^2+2r\cos(\theta)} \, d\theta \\
=&2(1-r)\int_{0}^{\pi-\delta} \frac{1}{1+r^2+2r\cos(\theta)} \, d\theta \\
=&2(1-r)\int_{0}^{\tan(\frac{\pi-\delta}{2})} \frac{}{1+r^2+2r\frac{1-t^2}{1+t^2}} \, \frac{2}{1+t^2}dt \\
=&\frac{4(1-r)}{(1+r)^2}\int_{0}^{\tan(\frac{\pi-\delta}{2})} \frac{1}{1+(\frac{1-r}{1+r})^2t^2 } \, dt \\
=& 4(1+r) \int_{0}^{\frac{1-r}{1+r}\tan(\frac{\pi-\delta}{2})} \frac{1}{1+u^2} \, du \\
=& 4(1+r) \arctan\left(\frac{1-r}{1+r}\tan\left(\frac{\pi-\delta}{2}\right)\right),
\end{align*}
finally the result follows from $\tan\left(\frac{\pi}{2}-x\right) \sim_0 \frac{1}{x}$.
\end{proof}

Then we are in position to prove that the harmonic extension of a map in  $H^\frac{1}{2}(\mathbb{S}^1,\mathbb{S}^1)$ is proper, more precisely

\begin{lem}
\label{Klem} 
Let $v\in H^{1/2}(\mathbb{S}^1,\mathbb{S}^1)$  and $\tilde{v}$ its harmonic extension on $\D$. We denote, for $z\in \D$, $I_z$ the subarc of $\mathbb{S}^1$ centred at $\frac{z}{\vert z\vert}$ of length $2(1-\vert z\vert)$, and for subarc $I\subset \mathbb{S}^1$ we set

$$v_I= \frac{1}{\vert I\vert} \int_I v\, d\theta .$$
 
Then 
\begin{equation}
\lim_{M\rightarrow + \infty} \lim_{r \rightarrow 1} \sup_{\vert z\vert\geq r } \vert \tilde{v}(z) - v_{MI_z}\vert =0.
\end{equation}

Moreover for every $\epsilon>0$ there exists $M>0$ and $r<1$ such that if 
$\vert z \vert \geq r$ there exists $S_z \subset M I_z$ such that

$$\vert S_z \vert \geq \frac{\vert MI_z\vert }{2} $$ 
and 
$$\vert  \tilde{v}(z) -v(x) \vert \leq \epsilon \text{ for all } x\in S_z,$$
in particular  $\tilde{v}$ is proper.
\end{lem}
The result and proof are sketch at beginning of section 8 of \cite{Bishop2}, we give details for the sake of completeness.
\begin{proof}
Let us remind that, $H^{1/2} \subset \mathrm{VMO} \subset \mathrm{BMO}$, see \S 1.2 \cite{BN}\footnote{In this reference it is proved on $\R$ but it is also true on $\mathbb{S}^1$ by conformal invariance of all those spaces, see corollary VI.1.3 \cite{Gar}}.
So we have, by Lemma 3 of \cite{BN}, that, for every $M>0$,
\begin{equation}
\label{lem3}
\lim_{\vert z\vert \rightarrow 1} \frac{1}{\vert M I_z \vert }\int_{MI_z} \vert v -v_{MI_z}\vert \, d\theta =0 \text{ uniformly in }z.
\end{equation}
Moreover, $\tilde{v}$ is given by the Poisson formula, 
$$\tilde{v}(z)=   \int_0^{2\pi} P_z(\theta) v(e^{i\theta})\, d\theta,$$
where $P_z(\theta)= \frac{1}{2\pi} \frac{1-\vert z\vert^2}{\vert z-e^{i\theta}\vert^2}$. Hence, let $M>0$, we have, using the fact that $\int_{\mathbb{S}^1} P_z \, d\theta =1$ and $P_z \geq 0$, that

$$
\vert \tilde{v}(z)- v_{MI_z}\vert \leq \int_{MI_z}  P_z \vert v- v_{MI_z}\vert \, d\theta  + \int_{(MI_z)^c}  P_z \vert v- v_{MI_z}\vert \, d\theta 
$$

Let $\epsilon>0$, thanks to \eqref{eqM} and the fact that $\vert v\vert =1$ and $\vert v_{MI_z}\vert \leq 1$, we have, for $M$ large enough, that

$$\lim_{\vert z\vert \rightarrow 1} \int_{(I_{MI_z})^c} P_z \vert v-v_{MI_z}\vert \,d \theta \leq \frac{\eps}{4}$$

Once $M$ is fixed, since $v$ is in VMO, we have, thanks to \eqref{lem3}), that 
$$\int_{I_{MI_z}} P_z \vert v-v_{MI_z}\vert \,d \theta \leq \frac{2M}{\pi \vert M I_z\vert } \int_{I_{MI_z}}  \vert v-v_{MI_z}\vert \,d \theta \leq \frac{\epsilon}{4}, 
$$
which proves the first part of the lemma. 

Moreover, $v$ satisfies the John-Nirenberg theorem, see (3') of \cite{JN}, there exist $\beta, C>0$ independent of $v$, such for any arc $I\subset \mathbb{S}^1$ we have
$$\frac{1}{\vert I \vert} \int_I e^{\frac{\beta \vert v-v_I \vert }{\kappa}} \, d\theta \leq C ,$$

where 
$$\kappa = \sup_{I'\subset I} \frac{1}{\vert I\vert } \int_I \vert v-v_I\vert \, d\theta ,$$
In particular, 
$$\vert \{ x \in MI_z \, \vert \, \vert v(x)-v_{MI_z}\vert \geq \frac{\epsilon}{2} \}\vert \leq C \vert MI_z\vert  e^{-\frac{\beta \epsilon}{2\kappa}}.$$
Hence by \eqref{lem3}, for $1-\vert z\vert $ small enought, there exists $x \in M_Iz$ such that 

\begin{equation}
\label{If}
\vert g_{M I_{z}} - g(x) \vert \leq \frac{\delta}{2}
\end{equation}

which achieves the proof of the lemma.
\end{proof}

\section{Note on the extended definition of $H^{\frac{1}{2}}\left(\Gamma\right)$} \label{appendixexplainH12}

In this section, we comment and justify the general definition of $H^{\frac{1}{2}}\left(\Gamma, \mathbb{S}^1 \right)$ (see definition \ref{H12Q}). 

First, let $\Omega\subset \R^2$ be a bounded open set and $f:\D \rightarrow \Omega$ be a uniformization. Then for any $u : \Omega \to \R^n$, we have
\begin{equation} \label{propeqofh10spaces} u\circ f \in H^1_0\left( \mathbb{D},\R^n \right) \Leftrightarrow u \in H^1_0\left( \Omega, \R^n \right) \hskip.1cm. \end{equation}
Notice that it is equivalent to $H^1_0\left(\mathbb{D},\R^n \right) = H^1_0\left(\mathbb{D}, \left\vert f' \right\vert^2 ,\R^n \right)$ by pullback of functions with $f$. In fact \eqref{propeqofh10spaces} is just an application of the conformal invariance of the Dirichlet energy and the classical Poincar\'e inequality on $\mathbb{D}$ and on $\Omega$:
$$ \int_{\mathbb{D}} \left\vert u\circ f \right\vert^2  \leq C_{\mathbb{D}} \int_{\mathbb{D}} \left\vert \nabla \left(u\circ f\right) \right\vert ^2 = C_{\mathbb{D}}  \int_{\Omega} \left\vert \nabla u \right\vert ^2 \leq C_{\mathbb{D}} \left\| u \right\|_{H^1\left( \Omega \right)}  $$
$$ \int_{\Omega} \left\vert u \right\vert^2  \leq C_{\Omega} \int_{\Omega} \left\vert \nabla u \right\vert ^2 = C_{\Omega} \int_{\mathbb{D}} \left\vert \nabla \left(u\circ f\right) \right\vert ^2  \leq C_{\Omega} \left\| u \circ f \right\|_{H^1\left( \mathbb{D} \right)} $$
for any $u\in \mathcal{C}^{\infty}_c(\Omega)$. Notice that the Poincar\'e inequality in $\Omega$ can also be seen as the classical Hardy inequality since $\sup_{z\in \mathbb{D}} \left\vert f'(z) \right\vert^2 \left(1-\left\vert z \right\vert\right)^2 < +\infty$ by Koebe's theorem.

Now, we would ideally like to have $ H^1\left(\mathbb{D},\left\vert f'(z) \right\vert^2 \right) = H^1\left(\mathbb{D} \right)$ too in order to give a sense to definition \ref{H12Q}. Indeed, we would like the function $g : \Gamma \to \mathbb{R}^n$ to be some trace of a $H^1\left( \Omega \right)$ function $u$. 
Pulling back to $\mathbb{D}$ by $f$, this means that $g \circ f : \Gamma \to \mathbb{R}^n$ is some trace of the function $u\circ f \in H^1\left(\mathbb{D},\left\vert f'(z) \right\vert^2 \right)$. However, $g \circ f : \Gamma \to \mathbb{R}^n$ is well-defined as a $H^{\frac{1}{2}}\left(\mathbb{S}^1\right)$ trace only if it is a trace of a $H^1\left(\mathbb{D}\right)$ function.

Then, we assume that $g \circ f$ admits some extension in $ H^1\left(\mathbb{D},\left\vert f'(z) \right\vert^2 \right) \cap H^1\left(\mathbb{D} \right)$ since those sets are not a priori equal. We then a priori ask more restrictions for being in $H^\frac{1}{2}(\Gamma, \R^n)$ than in definition \ref{H12Q}:
 
\begin{defi}[$H^\frac{1}{2}(\Gamma, \R^n)$ maps]
 Let $n\geq 1$, $\Gamma$ be a quasicircle enclosing a domain $\Omega$, and $f:\D \rightarrow \Omega$ a uniformization. We say that $g :\Gamma \rightarrow \R^n$ is in $ H^\frac{1}{2}(\Gamma,\R^n)$ if $g\circ f \in H^\frac{1}{2}(\mathbb{S}^1,\R^n)$ and if $g\circ f$ admits an extension $v \in H^1(\mathbb{D},\R^n)$ such that 
\begin{equation}
\label{DH12}
\int_{\mathbb{D}} \left\vert v \right\vert^2(z) \left\vert f' \right\vert^2(z) dz < +\infty
\end{equation} 
that is to say $v\circ f^{-1} \in H^1\left( \Omega,\mathbb{S}^1 \right)$.
\end{defi}

Once we know that such an extension $v$ exists, then by \eqref{propeqofh10spaces}, the set of $H^1\left(\mathbb{D}\right)$ extensions of $g\circ f$: $\{ v \} + H^1_0\left( \mathbb{D} \right)$ is equal to the pullback of the natural set of $H^1\left(\Omega \right)$ extensions of $g$: $\{ v\circ f^{-1} \} +  H^1_0\left( \Omega \right)$. This makes the definition consistent. In particular, the harmonic extension $v$ of $g\circ f$ satisfies $v\circ f^{-1} \in H^1\left(\Omega \right)$.

In the special case where $g$ is bounded, then \eqref{DH12} is automatically satisfied since $ H^1\cap L^{\infty}\left(\mathbb{D},\left\vert f'(z) \right\vert^2 \right) = H^1\cap L^{\infty}\left(\mathbb{D} \right)$ and the harmonic extension $v$ of $g\circ f$ is also bounded. 
This shows that definition \ref{H12Q} makes sense. 

\section{Wente type estimate}
We gives a proof of a Wente type lemma, which is inspired of Theorem A.4 of \cite{LL} (see also Theorem 1.100 \cite{Sem} and proof of Theorem 1.2 \cite{LP}).

\begin{lem}
\label{LLP}
There exists $C>0$ such that if $f\in L^1(\D)$ such that 

$$\sup_{\D}\vert f(x) (1-\vert x \vert)^2 \vert < +\infty,$$ then, there exists $\phi \in W_0^{1,2}(\D) \cap L^\infty(\D)$ such that 

$$\Delta \phi= f,$$
and
$$ \Vert \phi \Vert_{\infty} + \Vert \nabla \phi \Vert_2 \leq C (\Vert f\Vert_1 + \Vert f(1-\vert x\vert)^2 \Vert_\infty) .$$  
\end{lem}

\begin{proof}

It is classical that, by integration by part, the $L^\infty$-estimate implies the $L^2$-estimate. Hence we set
$$\phi(x)=\int_\D G(x,y) f(y)\, dy,$$
where $G$ is the Green function of the Laplacian. Then, we remark that the singularity of the Green function can be decomposed as a sum of bump functions with dyadic support, see Theorem A.4 of \cite{LL}: there exits $C>0$ such that 
$$-C \leq G(x,y)-l_x(y) \leq C,$$
with 
$$ l_x(y)=-\sum_{j=1}^\infty \Theta \left(\frac{2^j(x-y)}{1-\vert x\vert}\right),$$
where $\Theta\in C^\infty_c(\R, [0,1])$ is a  bump function, such as $\mathrm{supp}(\Theta)\subset [-\frac{1}{8},\frac{1}{8}]$ and $\Theta \equiv 1$ on $[-\frac{1}{16},\frac{1}{16}]$. Then
$$\vert \phi(x) \vert \leq C \Vert f\Vert_1 + \sum_{j=1}^\infty \int_\D \Theta \left(\frac{2^j(x-y)}{1-\vert x\vert}\right)f(y) \, dy .$$
Then we estimate each integral of the right hand 
\begin{equation}
\begin{split}
    \int_\D \Theta \left(\frac{2^j(x-y)}{1-\vert x\vert}\right) f(y) \, dy & \leq  \int_{B(x, 2^{-j-3}(1-\vert x\vert))}  f(y) \, dy\\
    & \leq    \Vert f(1-\vert x\vert)^2 \Vert_\infty  \int_{B(x, 2^{-j-3}(1-\vert x\vert))}  \frac{1}{(1-\vert y\vert)^2} \, dy \\
    & \leq \frac{\pi}{(2^{j+3}-1)^2}  \Vert f(1-\vert x\vert)^2 \Vert_\infty
\end{split}
\end{equation}
Finally we just have to sum this inequality to get the result.
\end{proof}

\bibliographystyle{plain}
\bibliography{biblio}

\end{document}